%% file: main_arxiv.tex
\documentclass[letter,12pt, reqno]{amsart}
\usepackage[margin=2.7cm]{geometry}
\usepackage{graphicx}%
\usepackage{multirow}%
\usepackage{amsmath,amssymb,amsfonts}%
\usepackage{amsthm}%
\usepackage{mathrsfs}%
\usepackage[title]{appendix}%
\usepackage{xcolor}%
\usepackage{textcomp}%
\usepackage{manyfoot}%
\usepackage{booktabs}%
\usepackage[ruled,vlined]{algorithm2e}
\usepackage{algpseudocode}%
\usepackage{listings}%
\usepackage{framed}

\usepackage{tikz}
\usepackage{pgfplots}

\usepackage{booktabs}
\usepackage{adjustbox} 
\usepackage{array}

\usepackage[colorlinks=true, allcolors=blue]{hyperref}

\theoremstyle{thmstyleone}%
\newtheorem{theorem}{Theorem}

\newtheorem{prop}[theorem]{Proposition}%

\theoremstyle{thmstyletwo}%
\newtheorem{example}{Example}%

\theoremstyle{thmstylethree}%

\newcommand{\R}{\mathbb{R}}

\newcommand{\bit}{\begin{itemize}}
\newcommand{\eit}{\end{itemize}}

\newcommand{\K}{\mathcal{K}}
\newcommand{\X}{\mathcal{X}}

\newcommand{\Hyp}{\mathcal{H}}
\newcommand{\ci}{c}
\newcommand{\cset}{\mathcal{C}}
\newcommand{\csetl}{\cset^+}
\newcommand{\csetr}{\cset^-}
\newcommand{\classes}{\mathcal{K}}
\newcommand{\ki}{k}

\newcommand{\sensel}{t}
\newcommand{\senser}{u}
\newcommand{\hi}{r}
\newcommand{\hset}{\mathcal{M}}
\newcommand{\di}{i}

\newcommand{\dset}{\mathcal{N}}
\newcommand{\maxnorm}{\kappa}
\newcommand{\hcoef}{a}
\newcommand{\hrhs}{b}
\newcommand{\cclass}{z}
\newcommand{\dclass}{y}
\newcommand{\assign}{v}

\newcommand{\err}{e}

\newcommand{\coefset}{\mathcal{A}}
\newcommand{\dsum}{\displaystyle\sum}

\newcommand{\sign}{\mathrm{sign}}
\newcommand{\exclude}[1]{}

\let\origmaketitle\maketitle
\def\maketitle{
	\begingroup
	\def\uppercasenonmath##1{} 
	\let\MakeUppercase\relax 
	\origmaketitle
	\endgroup
}

\begin{document}

\title[]{\Large A Unified Optimization Framework for  Multiclass Classification with Structured Hyperplane Arrangements}

\author[V. Blanco, H. Kothari, \MakeLowercase{and} J. Luedtke]{
{\large V\'ictor Blanco$^{\dagger}$, Harshit Kothari$^{\ddagger}$, and James Luedtke$^{\ddagger}$}\medskip\\
$^\dagger$Institute of Mathematics (IMAG), Universidad de Granada\\
$^\ddagger$Dept. Industrial \& Systems Engineering, University of Wisconsin-Madison\\
\texttt{vblanco@ugr.es}, \texttt{hkothari2@wisc.edu}, \texttt{jim.luedtke@wisc.edu}
}

\maketitle

\begin{abstract}
In this paper, we propose a new mathematical optimization model for multiclass classification based on arrangements of hyperplanes. Our approach preserves the core support vector machine (SVM) paradigm of maximizing class separation while minimizing misclassification errors, and it is computationally more efficient than a previous formulation. We present a kernel-based extension that allows it to construct nonlinear decision boundaries. Furthermore, we show how the framework can naturally incorporate alternative geometric structures, including classification trees, $\ell_p$-SVMs, and models with discrete feature selection. To address large-scale instances, we develop a dynamic clustering matheuristic that leverages the proposed MIP formulation. Extensive computational experiments demonstrate the efficiency of the proposed model and dynamic clustering heuristic, and we report competitive classification performance on both synthetic datasets and real-world benchmarks from the UCI Machine Learning Repository, comparing our method with state-of-the-art implementations available in \texttt{scikit-learn}.
\end{abstract}

\section{Introduction}

Supervised classification is one of the fundamental problems in data science that aims to find algorithmic tools to decide the classification of a set of objects into a set of classes. The \emph{classifier} is derived from a set of training samples where the information of a set of individuals and their class is provided. The main challenge in supervised classification is then to construct classification rules able to capture, not only the known classes of the training sample observations, but also of out-of-sample observations. The applications of this type of tools can be found everywhere -- writing recognition~\cite{writing}, finance~\cite{credit,insurance}, and health\cite{cancer,cleveland} are just a few examples.

Different tools have been developed to construct classifiers. Most of them are optimization-based methodologies that derive the classifier as the solution of a mathematical optimization problem of different nature. Support Vector Machines~\cite{svm,bpr20}, Boosting~\cite{freund1997decision}, and Logistic Regression ~\cite{collins2002logistic} use of convex optimization tools, whereas others use tools from linear programming~\cite{freed1986evaluating}. More recently, discrete optimization has been used to derive reliable, interpretable, and robust classifiers. The theoretical advances in mixed integer linear and nonlinear optimization, together with the efficient implementations in the off-the-shelf optimization solvers, has made it possible to use discrete optimization for supervised classification with different characteristics.  Examples include classification and feature selection~\cite{baldomero2020tightening,baldomero2021robust,benitez2019cost,labbe2019mixed}, tree-shaped classification \cite{bertsimas2017optimal,aghaei2020learning,gunluk2018optimal,bmz-cpaior22,blanco2022forests,biggs2025tight}, robust classification~\cite{blanco2022mathematical,blanco2022robust}, among many others.

In this paper we analyze an extension of  the support vector machine (SVM) that both allows to classify multiclass instances and also detect patterns in geometrically intricate instances. SVM is one of the most popular optimization based tools for supervised binary classification~\cite{svm}. The linear and bi-class SVM methodology is based on finding a separating hyperplane that both maximizes the separation between classes and minimizes the sum of the training misclassifications errors. A significant advantage of this tool is that, using the convexity properties of the problem one can compute, via the \emph{kernel trick}, non linear separators with the same computational effort than a linear one. Nevertheless, this tool is designed only for bi-class instances, although some efforts have been paid to extend SVMs to multi-class instances. The most natural approaches, although suboptimal, are the so-called one-versus-all and one-versus-one, which are based on solving, independently, small instances of binary SVMs independent for different subsets of the training sample. Although one can extend the properties of SVM to these extensions, it has been shown that they allow to provide good classifiers only for instances which are geometrically simple. In another line of extensions of SVM to multiclass instances, in \cite{cs,ww,llw} the authors propose different global approaches to model the multiclass SVM problem as a single optimization problem. Nevertheless, as already noticed by  \cite{BlancoJaponPuerto_ADAC2020}, these extensions are based on mathematical optimization models which do not keep the main paradigm of the optimization problem behind the SVM, which is to construct linear structures that both minimize the misclassifications errors but also that maximize the separation between the different classes. In binary SVM the separation between the classes is measured by means of the Euclidean width of the margin, the space between the separated classes, which by \cite{mangasarian1999arbitrary} can be quantified as twice the inverse of the norm of the coefficients of the separating hyperplane. Thus, maximizing such a width, for a single SVM-based hyperplane, it is equivalent to minimize the norm of those coefficients. In the multiclass extensions in \cite{cs,ww,llw}, the optimization models are based on minimizing the sum of the norms of the different hyperplanes that are to be found, but this amount does not represent the actual separations between the classes. In contrast,  \cite{BlancoJaponPuerto_ADAC2020} propose a mixed integer nonlinear programming model to correctly maximize the minimum width of the different margins. The main bottleneck of this approach is the limited size of the training instances that can be optimally solved, since the models considers different discrete decisions in the models as the assignments of cells to classes, assignments of observations to cells and classes, different types of missclassification errors, etc. 

In this paper, we propose an alternative mathematical optimization model for determining hyperplanes that separate classes, which is computationally more efficient than previously proposed methodology. We investigate techniques to handle the symmetry inherent in the model and derive a general framework for computing nonlinear separators for a given training set, extending the kernel trick for SVM to our model. Although our methodology is primarily designed to extend SVM-based classifiers to multiclass problems, we show that it can be naturally adapted to incorporate other structures for geometrically separating classes, such as classification trees, $\ell_p$-SVMs, or SVM-based hyperplanes with discrete feature selection.

The proposed approaches are evaluated through an extensive set of computational experiments. First, we compare the computational performance of our proposed mathematical optimization model with the previous model, as well as existing methods from the literature for multiclass classification, on a collection of synthetic instances. Next, to allow our method to scale up to instances with many data points, we introduce a math-heuristic that uses our formulation on smaller cluster-based data sets that are constructed iteratively. We conduct experiments to assess the computational and predictive performance of the this method. Finally, we apply our method to real-world datasets from the UCI Machine Learning Repository and compare its \textit{classification} performance with that of the existing implementations available in \texttt{scikit-learn}.
 
The rest of the paper is organized as follows. In Section \ref{sec:1}, we introduce the multiclass classification problem under study, detailing all the elements involved. Section \ref{sec:2} presents the mathematical optimization model we propose for the problem, describe techniques for addressing symmetry in the formulation, and explain how the formulation can also be used to construct multivariate classification trees. In Section \ref{sec:3}, we develop a kernelized version of our model, which allows the construction of nonlinear decision boundaries in certain cases. Finally, Section \ref{sec:5} reports the results of our computational experiments.

\section{Multiclass Classification with Structured Hyperplane Arrangements}\label{sec:1}

In this section, we describe the problem under study and establish the notation used throughout the paper. The presentation is organized around the four main components of the problem: the input data, the arrangement of hyperplanes, the classification errors, and the resulting classification rule.\\

\noindent{\bf Training Data and Classification Rule}\\

We are given a training dataset $\X = \{(x_1, \dclass_1), \ldots, (x_n,\dclass_n)\} \subseteq \R^d \times \{1, \ldots, K\}$. We denote by $\dset=\{1, \ldots, n\}$ the index set for the observations in that sample and by $\classes=\{1, \dots, K\}$ the set of classes. For each $i\in \dset$, $x_i$ represents the values of the $d$ different features for such an observation, whereas $\dclass_i \in \classes$ represents its class.

The goal of a classification rule is to construct a classifier, based on the training sample, $\Psi_\X: \R^d \rightarrow \classes$, able to correctly classify out-of-sample data.

We analyze classifiers that are based on constructing a geometrical partition of the feature space $\R^d$ and assigning each of the resulting cells a class in $\classes$. In particular, $\Psi_\X(x) = k$, if $x \in R_k$, for $k\in \classes$, where $R_1, \ldots, R_K \subseteq \R^d$ are such that such $R_i \cap R_j = \emptyset$ for all $i, j=1, \ldots, K$ and $\bigcup_{i=1}^K R_i = \R^d$. 
Most popular classification methods in the literature are based on this idea, as SVM\cite{svm}, CART\cite{CART}, OCT\cite{bertsimas2017optimal}, LDA~\cite{fisher1936use}.

We assume that the partition defining the classifier is induced by an arrangement of a given number, $m$, of hyperplanes, $\Hyp=\{H_1, \ldots, H_m\}$, where $H_\hi= \{z \in \R^d: \hcoef_\di^{\top} z + \hrhs_\di=0\}$ for each $\hi \in \hset=\{1, \ldots, m\}$. The regions defining the partition results from the set of polyhedral cells obtained with the arrangements of hyperplanes $\Hyp$. Each of the cells induced by the arrangement can be uniquely identified with a sign pattern in $\{-1,1\}^{m}$, the sign when evaluating the points in the cell at each of the affine functions defining the hyperplanes (equivalently, the side of the hyperplane where the cell is located). We denote by $\cset =\{-1,1\}^m$ the set of cells induced by the arrangements, and by $\ci_\hi$ the sign in such a pattern for cell $\ci$ with respect to hyperplane $\hi\in M$. We say that observation $x \in \R^d$ is in the left (resp. right) -branch of the $\di$-th hyperplane $H_\hi = \{z \in \R^d: \hcoef_\hi^{\top} z + \hrhs_\hi = 0\}$ if $ \hcoef_\hi^{\top} z + \hrhs_\hi \geq 0$ (resp. $\hcoef_\hi^{\top} z + \hrhs_\hi < 0)$. 

Additionally, we denote by $\csetl_{\hi}$ the set of cells $\ci \in \cset$ that are associated to the left-branch of hyperplane $\hi \in \hset$, and by $\csetr_{\hi}$ the set of cells $\ci \in \cset$ that are associated to the right-branch of split $\hi \in \hset$.\\

\noindent{\bf Separating Hyperplanes}\\

We assume that the coefficients and intercept of the hyperplanes belong to the prespecified set $\coefset \subset \R^d \times \R$. This set represents generic restrictions on the coefficients and intercept that can be selected to define the hyperplanes. Different choices for $\coefset$ lead to different models. 

For instance, in SVM-based separators, the hyperplanes must assure a minimal separation between the classes that are separated by the hyperplanes. Given a hyperplane $H= \{z \in \R^d: \hcoef^{\top} z + \hrhs =0\}$, the $p$-norm separation (for $p\geq 1$) is defined by the expression $2/\|\hcoef\|_q$, where $q$ is such that $\frac{1}{p}+\frac{1}{q} = 1$. Thus, constraint $\|\hcoef\|_q \leq \kappa$ for a given $\kappa >0$ assure a minimal separation of $2/\kappa$. This constraint for the coefficient of the hyperplanes induces the set $\coefset_{{\rm SVM}_p} = \{(\hcoef,\hrhs) \in \R^d \times \R: \|\hcoef\|_q \leq \kappa\}$. Thus, for such a choice of $\coefset$, the parameter $\kappa$ can be interpreted as a hyperparameter of the model that controls the size of the margin around the hyperplanes, in a similar manner as the coefficient on the norm of the hyperplane in the objective of a standard SVM model. The particular case of $p=2$ (and then  $q=2$) corresponds to a 2-norm margin normalization, and will allow, as we will see, the application of the kernel trick to create nonlinear separating curves. The case $p=1$ corresponds to a 1-norm margin normalization, and may be useful for encouraging selection of coefficient vectors $\hcoef$ that are sparse.

Another possible choice is the one that allows only axis-aligned hyperplanes, i.e. $\coefset_{0}= \{(\hcoef,\hrhs) \in \R^d \times \R: \|\hcoef\|_0 =1\}$, which is convenient in interpretable machine learning (the classifier uses only one feature to separate the classes, so it is extremely simple to explain) and in feature selection, since it forces the model to pick exactly one relevant feature among many.

A less explored choice could be  $\coefset_{\rm Disc} = \{ (\hcoef_\ell, \hrhs_{\ell}): \ell=1,\ldots,L\}$. This choice corresponds to the case where a fixed candidate list of hyperplanes is given. For example, this would be the case where separation is limited to univariate branching (as in the previous case) and a finite set of break points for each component of $x_{\di}$ is pre-selected.\\

\noindent{\bf Misclassification Errors}\\

The goal of our approach is to locate the $m$ hyperplanes in the feature space such that they minimize the overall missclassification errors of the points in the training sample with respect to the class assigned to each cell in the arrangement.

Then, our model will compute both the $m$ hyperplanes and an assignment to each cell $c\in \cset$ a class $k(c)$ in $\cset$ such that the misclassification errors are minimized. 

To this end, for each training observation, $x$, and each hyperplane, $H=\{z: \hcoef^{\top} z + \hrhs =0\}$, in the arrangement,  we calculate how far is the observation with respect to the \textit{closer} side of the hyperplane that was assigned to its actual class. Specifically, the error is computed as $e = \max \{0, 1 - \sign(\hcoef^{\top} x + \hrhs) (\hcoef^{\top} x + \hrhs)\}$. This error is summed over the different hyperplanes defining a cell. In addition, since in our model multiple cells may be assigned as the class of $x$, this error is defined to be the smallest error among cells that match the class of $x$.

This definition of misclassification errors is aligned with the ``margin'' error philosophy of support vector machines, differentiating it, for example, from an error metric which instead counts the number of training points which are misclassified.\\

\noindent{\bf Classification Rule}\\

Once the coefficients and intercepts of the $m$ hyperplanes are computed, $(\hcoef_r, \hrhs_\hi)$, for $r\in \hset$, as well as the assignment to each cell in $\cset$ to a class through the mapping $z: \cset \rightarrow \classes$, we derive the decision rule for the feature space, $\Psi: \R^d \rightarrow \classes$ based on these values. That is, for a given $x \in \R^d$, the sign-pattern of $x$, with respect the $m$ hyperplanes is computed, $c(x)$, and the assignment based on $z$ is then applied, being $\Psi(x) = z(c(x))$. Nevertheless, since the hyperplanes and the assignment will be computed based on the training sample, the assignment of cells with no training observations inside is arbitrary, and in that case, a more accurate assignment can be done by allocating $x$ to the cell and class with the smallest global misclassification error. It can be done by computing a cell in $\cset$ minimizing the misclassification error, that is:
$$
\bar c \in \arg\min_{c\in \cset} \sum_{r\in M} c_r (\hcoef_\hi^{\top} x + \hrhs_\hi),
$$
and assign $\Psi(x) = z(\bar c)$.

With all these ingredients, we derive a methodology to construct, by means of a mathematical optimization model, accurate classification tools able to capture intricate multiclass instances. The key idea of our approach is illustrated in the following example.

\begin{example}\label{ex:1}
    Let us consider the two-dimensional training instance shown in Figure~\ref{ex:data}, where the three classes are highlighted in different colors. Figure~\ref{ex:sol} presents three different solutions obtained to classify this dataset. Each region of the plane is filled with the color of the class assigned to it.

In the left plot, we show the solution obtained with our methodology. Here, two hyperplanes generate four polyhedral regions in the plane, each assigned to a class, which in this case perfectly match the true classes of the given instance.

In the center plot, we display the solutions obtained with OVO, OVR, and CS approaches for linear separation of the classes. In this case, these popular SVM-based heuristics fail to capture the geometric structure of the data.

Finally, in the right plot, we show the solutions of the same three approaches when using RBF kernels (nonlinear separators). This is the only way to correctly classify the data, although it requires tuning additional hyperparameters and increases the risk of overfitting.

    \begin{figure}[h]
    \centering
 \fbox{\includegraphics[width=.7\linewidth]{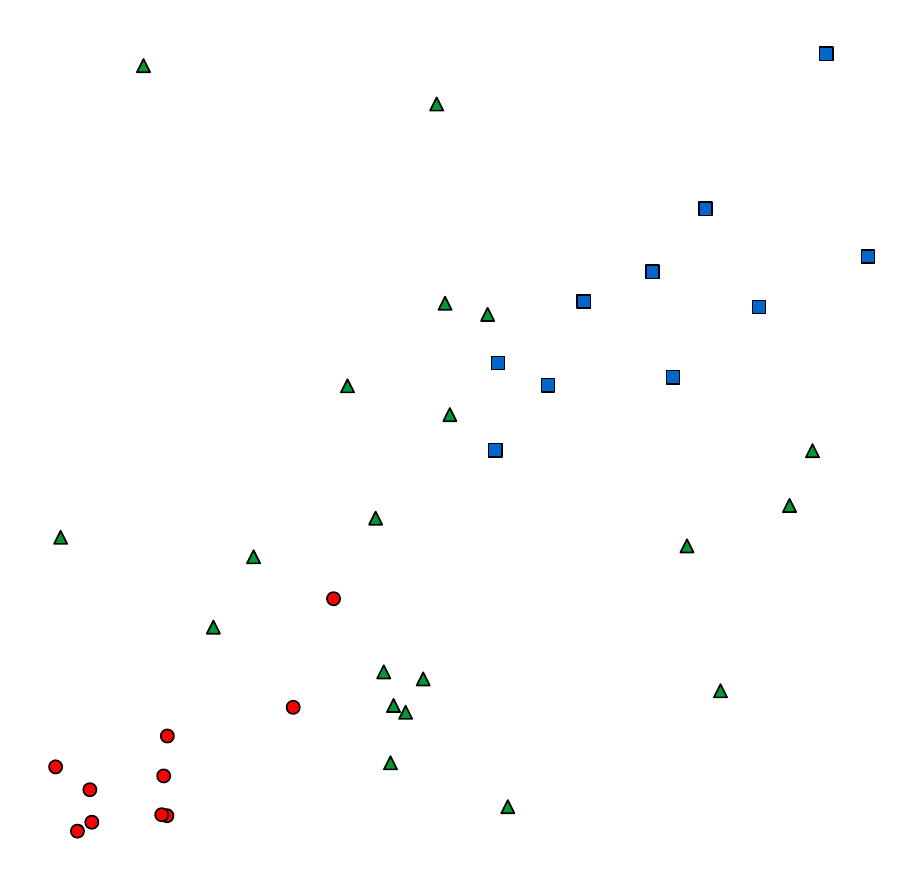}}
    \caption{Dataset to illustrate the methodology.\label{ex:data}}
    \end{figure}

    \begin{figure}[h]
    \centering
 \includegraphics[width=.32\linewidth]{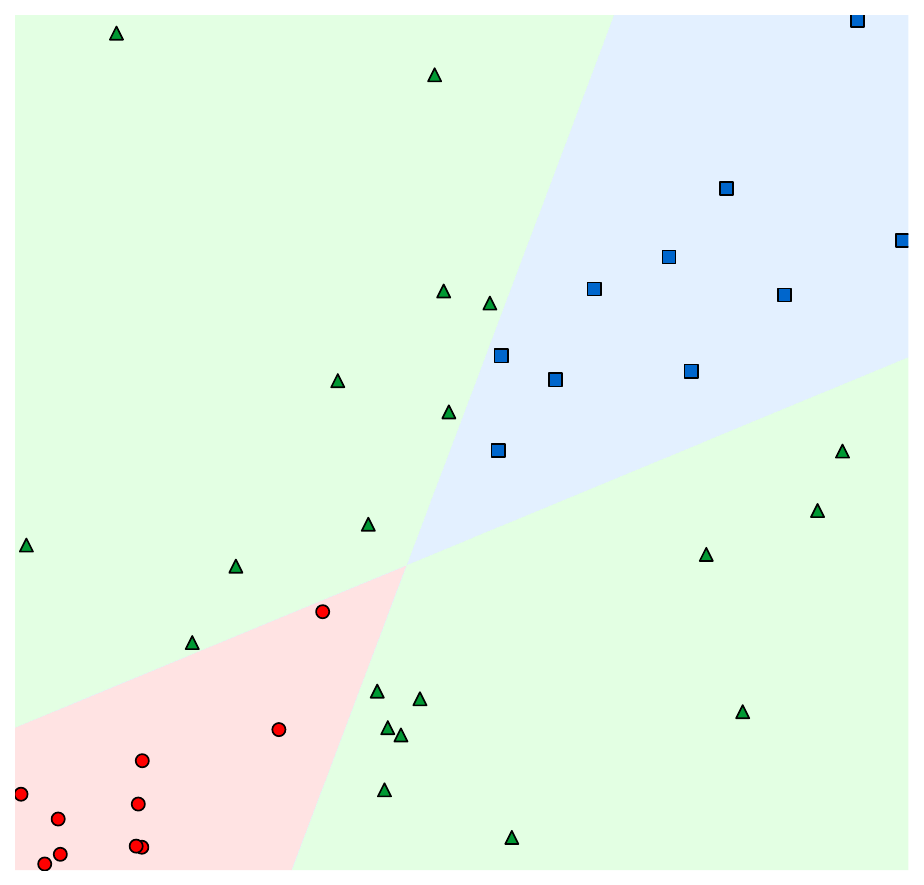}~ \includegraphics[width=.32\linewidth]{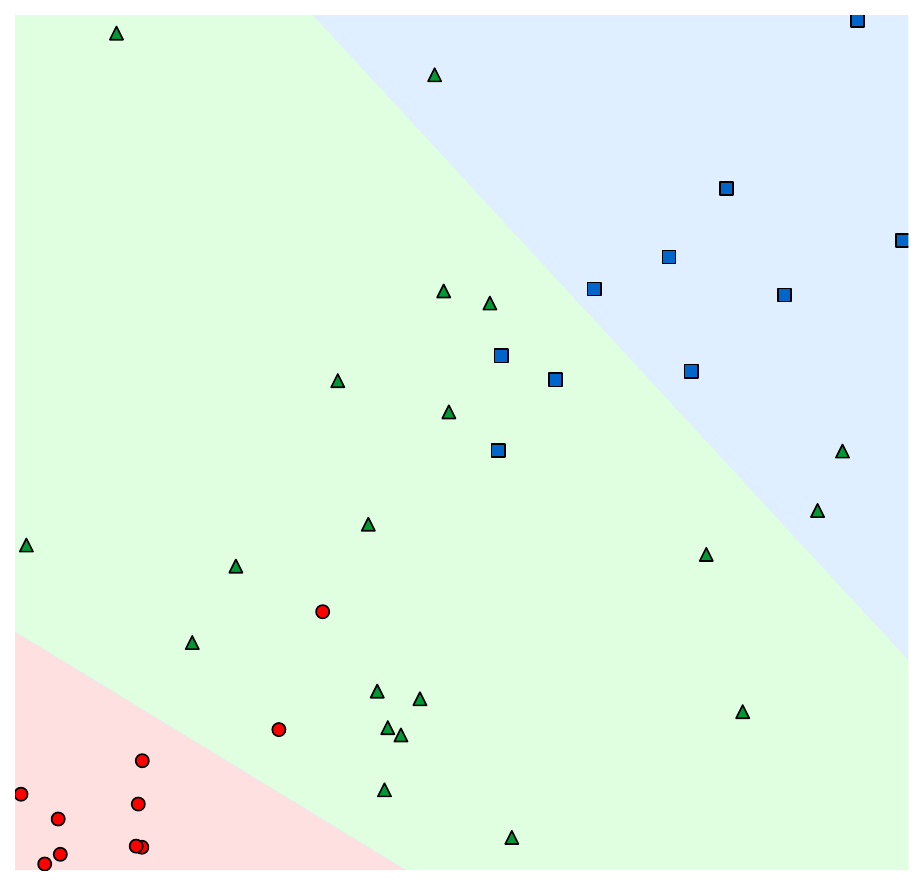}~ \includegraphics[width=.32\linewidth]{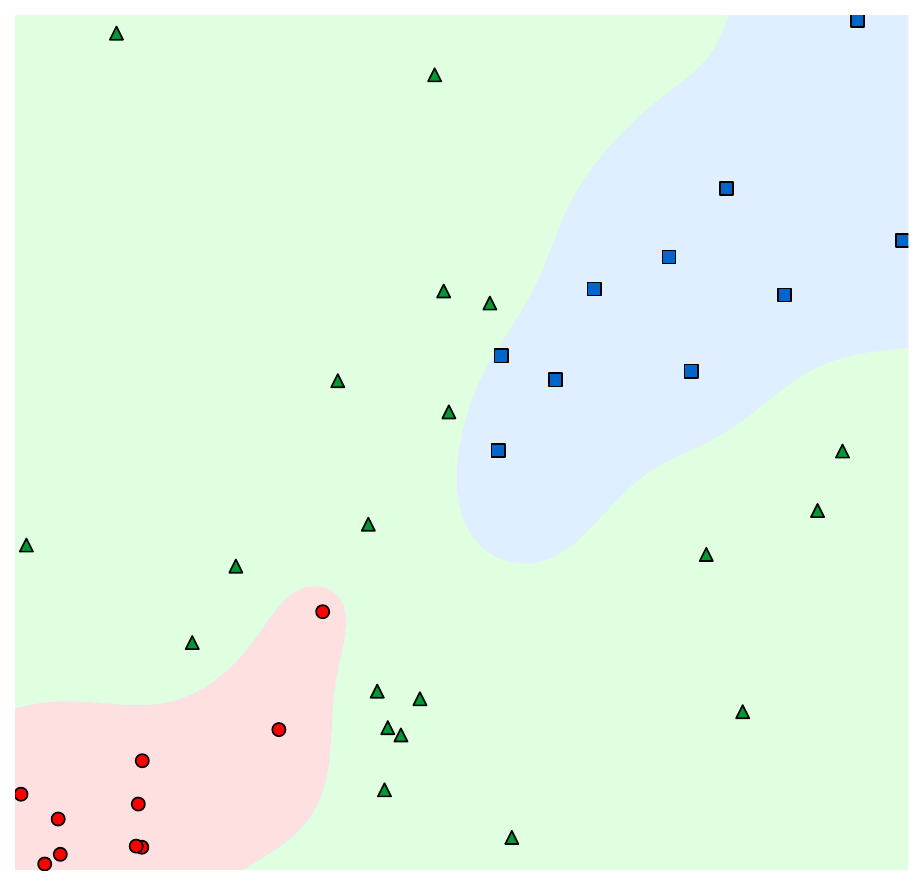}
    \caption{Solutions obtained with our methodology (left), classical linear SVM-based  methods (center), and rbf SVM-based methods .\label{ex:sol}}
    \end{figure}
    
\end{example}

\section{Mathematical Optimization Models}\label{sec:2}

In this section we provide two mathematical optimization formulations to construct an arrangement of hyperplanes that minimizes the sum of hinge losses over the training data.

\subsection{Prior Formulation}

Before presenting our proposal, we summarize the model introduced in \cite{BlancoJaponPuerto_ADAC2020}, since both approaches share the same underlying philosophy, and this comparison provides a fair basis for evaluating the advantages of our formulation. 

Table~\ref{vars:BJP20} lists the decision variables used in that model. The formulation determines the coefficients and intercepts of the hyperplanes, the assignment of training observations to classes, and several auxiliary variables that capture different types of misclassification errors (such as in-margin and out-margin errors), the correctness of the classification, and the identification of representative elements within each cell.

\begin{table}[h]
\begin{tabular}{|cp{9cm}|}\hline
\multicolumn{2}{|c|}{Continuous Variables}\\\hline
$\hcoef_r \in \R^d$, $\hrhs_\hi \in \R$ & Coefficients and intercept of the $r$th hyperplane, $\forall r \in \hset$.\\
$e^O_{ir} \geq 0$ & Out-margin error of observation $i$ w.r.t. hyperplane $r$, $\forall i \in \dset, r \in \hset$.\\
$e^I_{ir} \geq 0$ & In-margin error of observation $i$ w.r.t. hyperplane $r$, $\forall i \in \dset, r \in \hset$.\\
$\theta \geq 0$ & $\max_{r\in M}\|\hcoef_r\|^2$.\\\hline
\multicolumn{2}{|c|}{Binary Variables}\\\hline
$\sensel_{\hi\di}$ & {$1$ if $\hcoef_r^{\top} x_i+\hrhs_\hi \geq 0$, $\forall i \in \dset, r \in \hset$.}\\
$z_{i\ki}$ &  {$1$ if $i$ is assigned to class $s$, $\forall i \in \dset, \ki \in \classes$.}\\
$\xi_i$ &  {$0$ if $i$ is well-classified, $\forall i \in \dset$.}\\
$h_{ij}$ & {$1$ if $x_j$, which is well classified and verifies $y_j=y_i$, is the representative of $x_i$  in its closest cell  through hyperplanes, $\forall i, j \in \dset$.}\\\hline

\end{tabular}
\caption{Decision Variables in the model proposed in \protect\cite{BlancoJaponPuerto_ADAC2020}\label{vars:BJP20}}
\end{table}

Then, the model proposed in \cite{BlancoJaponPuerto_ADAC2020} combines the maximization of the minimum hyperplane margin (among the $m$) with the minimization of two types of errors: the so-called out-margin errors ($d$-variables), which account for observations that, being outside the margin of the hyperplane, are incorrectly classified, and the in-margin errors ($e$-variables), which penalize being inside the margins of the different hyperplanes. The trade-off between these three components is controlled by the parameters $C_1$ (cost by unit of in-margin error) and $C_2$ (cost by unit of out-margin error) that are tuned in the training phase of the approach.

The complete model that the author propose reads as follows:
\begin{subequations}
\begin{align}
\min & \;\; \theta+ C_1 \dsum_{i\in \dset} \dsum_{r\in M} e^O_{ir}+   C_2\dsum_{i\in \dset} \dsum_{r\in M} e^I_{ir} \label{model}\tag{${\rm BJP20}$}\\
\mbox{s.t. } & \theta \geq \|\hcoef_r\|^2, &\forall r\in M,\label{q1}\\
&\hcoef_{r}^{\top}x_i + \hcoef_{r_0} \geq - T (1-\sensel_{\hi\di}), &\forall i\in \dset, r \in M,\label{t+}\\
& \hcoef_{r}^{\top}x_i + \hcoef_{r_0}  \leq T \sensel_{\hi\di}, &\forall i\in \dset, r \in M,\label{t-}\\
&\dsum_{k \in \classes} z_{ik}=1, &\forall i\in \dset,\label{q2}\\
& \|z_i - z_j\|_1 \leq 2 \|t_i-t_j\|_1, &\forall i, j\in \dset,\label{q3}\\
& \dsum_{j \in \dset:\atop y_i=y_j} h_{ij}=1, &\forall i\in \dset,\label{q6}\\
& \xi_j+h_{ij} \leq 1,  &\forall i, j\in \dset (y_i=y_j),\label{q7}\\
& h_{ii} = 1- \xi_i,   &\forall i\in \dset,\label{q7a}\\
&\hcoef_{r}^{\top} x_i + \hrhs_\hi \geq 1-e^I_{ir}- T\;(3-\sensel_{\hi\di}-t_{rj}-h_{ij}), &\forall r \in M,\label{q4}\\
& \hcoef_{r}^{\top} x_i +\hrhs_\hi \leq -1 +e^I_{ir}+ T\; (1+\sensel_{\hi\di}+t_{rj}-h_{ij}),& \forall r\in M,\label{q5}\\
& e^O_{ir} \geq 1- \hcoef^{\top}_r x_i -\hrhs_\hi - T\,(2+\sensel_{\hi\di}-t_{rj}-h_{ij}),& \forall i,j\in \dset (y_i=y_j), r\in M,\label{q8}\\
& e^O_{ir} \geq 1+ \hcoef^{\top}_r x_i + \hrhs_\hi - T\,(2-\sensel_{\hi\di}+t_{rj}-h_{ij}),& \forall i,j\in \dset (y_i=y_j), r\in M,\label{q9}\\
& \hcoef_r \in \R^p, \hrhs_\hi \in \R,&\forall r \in M,\nonumber\\
& e^O_{ir}, e^I_{ir} \geq 0,  \sensel_{\hi\di} \in \{0,1\}, &\forall i\in \dset, r\in M,\nonumber\\
& h_{ij} \in \{0,1\}, &\forall i, j \in \dset,\nonumber\\
& z_{ik} \in \{0,1\},& \forall i\in \dset, k\in \classes,\nonumber\\
& \xi_i \in \{0,1\}, & \forall i\in \dset, \nonumber
\end{align}
\end{subequations}
 where $T$ is a big enough constant that can be accurately estimated based on the training sample. There, as mentioned, the objective function combines the margin, the in-margin and the out-margin errors. Constraints \eqref{q1} assure the encoding in $\theta$ of the hyperplane with minimum margin. Constraints \eqref{t+} and \eqref{t-} correctly define the $t$-variables. Constraints \eqref{q2} assure that a single class is assigned to each training observation. In Constraints \eqref{q3}, the assignment of the same class to the same sign-patterns with respect to all hyperplanes is ensured. Constraints \eqref{q6} impose a single assignment between observations belonging to the same class. Constraints \eqref{q7} avoid choosing wrong-classified representative observations. The set of constraints \eqref{q7a} enforces well-classified observations to be represented by themselves. Constraints \eqref{q4}-\eqref{q9} assure the adequate definition of the in-margin and out-margin errors.

Note that the solution of the above model provides the \textit{optimal} coefficient and intercept of the $m$ hyperplanes and the assignment of each observation to a class. Thus, after solving the problem, in order to classify out-of-sample observations that are not in cells with training observations inside.

The main drawbacks of the model in \cite{BlancoJaponPuerto_ADAC2020} is its number of binary variables, $n^2 + nm + nk + n$, which can be prohibitive for large datasets, and the two hyperparameters, $C_1$ and $C_2$, that must be tuned in the training phase, apart from those that are induced by the kernel function, in case nonlinear separators are desired. In what follows we derive an alternative formulation for the problem that only requires a single hyperparameter, and that for a small number of hyperplanes (as is typical), requires significantly fewer binary variables, allowing solving the problem for larger datasets.\\

\subsection{New Formulation}

We propose an alternative mathematical optimization model for the problem, which follows the same paradigm as the previous one but exhibits improved computational performance and naturally extends to other structural variants (as discussed in Sections~\ref{sec:tree} and~\ref{sec:3}). As in the earlier formulation, the model determines the coefficients and intercepts of the $m$ hyperplanes; however, instead of assigning training points directly to classes, it assigns the cells induced by the arrangement of hyperplanes to classes. To this end, we introduce the following sets of decision variables:\\

\begin{framed}
\noindent$\hcoef_\hi \in \R^d, \hrhs_\hi \in \R$: Coefficients and intercept of the $\hi$ hyperplane, $\forall \hi \in \hset$.
\\

\noindent$
\cclass_{\ci \ki} \in \{0,1\}$:  Takes value $1$ if cell $\ci$ is assigned to class $\ki$, equal to $0$ otherwise, $\forall \ci \in \cset, \ki \in \classes$. \\

\noindent$
\assign_{\di \ci} \in \{0,1\}$: Takes value $1$ if data point $\di$ is assigned to cell $\ci$, and $0$ otherwise, $\forall \di \in \dset, \ci \in \cset$.\\

\noindent$\sensel_{\hi \di} \in \{0,1\}$: Takes value $1$ if data point $\di$ is assigned to the left branch of split $\hi$ and equal to $0$ otherwise, $\forall \hi \in \hset, \di \in \dset$.
\\

\noindent$\senser_{\hi \di} \in \{0,1\}$: Takes value $1$ if data point $\di$ is assigned to the right branch of split $\hi$ and equal to $0$ otherwise, $\forall \hi \in \hset, \di \in \dset$.
\\

\noindent$\err_{\di \hi} \in \R_+$: Hinge loss error of observation $i$ with respect to split $\hi$, $\forall \hi \in \hset, \di \in \dset$. 
\end{framed}

The variables introduced above are linked through the following linear constraints, which ensure their correct definition and guarantee the proper functioning of the model.

\begin{itemize}
\item Each training observation is assigned to a single sign-pattern with respect to the hyperplanes:
$$
\sum_{\ci \in \cset} \assign_{\di \ci} = 1,  \forall \di \in \dset.
$$
\item Each cell must have exactly one class assigned to it:
$$
\sum_{\ki \in \classes} \cclass_{\ci \ki} = 1, \forall \ci \in \cset.
$$
\item Each training point $\di$ must be assigned to a cell whose assigned class matches the true class $\dclass_{\di}$ of $\di$: 
$$
\assign_{\di \ci} \leq \cclass_{\ci \dclass_{\di}},  \forall \di \in \dset, \ci \in \cset .
$$
That is, in case cell $\ci$ is not assigned to the class of observation $\ci$, the observation is not allowed to be allocated to that cell. Note that in case the observation is geometrically positioned in a cell whose assigned class is not its class, a misclassification error will be accounted via the variables $\err_{\di \hi}$..
\item The following constraints assure the variables $\sensel_{\hi \di}$ and $\senser_{\hi \di}$ correctly record whether data point $\di$ is on the left or right side of hyperplane $\hi$, as determined by the cell the data point is assigned to:
\begin{align*}
\sensel_{\hi \di} &= \sum_{\ci \in \csetl_{\hi}} \assign_{\di \ci}, \forall \di \in \dset, \hi \in \hset, \\
\senser_{\hi \di} &= \sum_{\ci \in \csetr_{\hi}} \assign_{\di \ci}, \forall \di \in \dset, \hi \in \hset.
\end{align*}
\item The following constraints record the misclassification errors for each data point and each hyperplane:
\begin{align*}
\err_{\di \hi} &\geq 1 - (\hcoef_{\hi}^{\top} x_i + \hrhs_{\hi}) - \Delta_{\di} (1-\sensel_{\hi \di}),  \forall \di \in \dset, \hi \in \hset, \\
\err_{\di \hi} &\geq 1 + (\hcoef_{\hi}^{\top} x_i + \hrhs_{\hi}) - \Delta_{\di}(1- \senser_{\hi \di}), \forall \di \in \dset, \hi \in \hset, \\
\err_{\di \hi} &\geq 0, \forall \di \in \dset, \hi \in \hset.
\end{align*}
Here, $\Delta_\di$ is a large enough constant to be specified later. The first two constraints account  for the right and left branch misclassification errors with respect to all the $m$ hyperplanes. For each hyperplane, in case it is positioned in the left (resp. right) branch, the error measure is calculated with respect to the margin induced by the left (resp. right) margin of such an hyperplane. The overall sum of these errors will be minimized in the objective function of the model, and then reaching with equality one of the two above constraints. Since a data point $\di$ will have either $\sensel_{\hi \di} = 0$ or $\senser_{\hi \di} = 0$, only one of these constraints will be active. Furthermore, if the right-hand side or the associated constraint does not exceed zero, then the error variable $\err_{\di \hi}$ will only be constrained by $\err_{\di \hi} \geq 0$ (and hence will equal zero in an optimal solution).

\item Constraints on hyperplane coefficients and intercepts:
$$
(\hcoef_\hi, \hrhs_\hi) \in \coefset, \forall \hi \in \hset.
$$
Possible options for the form of the set $\coefset$ are discussed in Section~\ref{sec:1}. For norm-constrained sets (as motivated by SVM margin rationale), they lead to conic constraints (second-order cones or $p$-order cones), whereas for discrete sets the conditions can be imposed through linear constraints with additional binary variables (to determine which coefficients are selected from the potential set). 
\end{itemize}

Summarizing all the above, the resulting mathematical optimization formulation that we propose for Multiclass Classification with Hyperplane Arrangements is the following:
\begin{subequations}
\label{eq:form}
\begin{align}
 \min &\;\; \sum_{\hi \in \hset} \sum_{\di \in \dset} \err_{\di \hi} \\
    \mbox{s.t. } & \sum_{\ci \in \cset} \assign_{\di \ci} = 1, && \forall \di \in \dset, \label{f2:1}\\
    & \sum_{\ki \in \classes} \cclass_{\ci \ki} = 1, && \forall \ci \in \cset, \label{f2:2}\\
&  \assign_{\di \ci} \leq \cclass_{\ci \dclass_{\di}}, && \forall \di \in \dset, \ci \in \cset , \label{f2:3}\\
& \sensel_{\hi \di} = \sum_{\ci \in \csetl_{\hi}} \assign_{\di \ci} && \forall \di \in \dset, \hi \in \hset , \label{f2:4a}\\
& \senser_{\hi \di} = \sum_{\ci \in \csetr_{\hi}} \assign_{\di \ci}, && \forall \di \in \dset, \hi \in \hset,  \label{f2:4b}\\
& \err_{\di \hi} \geq 1 - (\hcoef_{\hi}^{\top} x_i + \hrhs_{\hi}) - \Delta_{\di} (1-\sensel_{\hi \di}),  &\quad& \forall \di \in \dset, \hi \in \hset,  \label{f2:5}\\
& \err_{\di \hi} \geq 1 + (\hcoef_{\hi}^{\top} x_i + \hrhs_{\hi}) - \Delta_{\di}(1- \senser_{\hi \di}), &\quad& \forall \di \in \dset, \hi \in \hset,  \label{f2:6}\\
& \err_{\di \hi} \geq 0, && \forall \di \in \dset, \hi \in \hset, \\
& (\hcoef_{\hi},\hrhs_{\hi}) \in \coefset, && \forall \hi \in \hset, \label{a:set}\\
& \assign_{\di \ci} \in \{0,1\}, \cclass_{\ci \ki} \in \{0,1\}, \label{vz:binary}\\
& \sensel_{\hi\di}, \senser_{\hi\di} \in \{0,1\}, \forall \hi\in \hset, \di \in \dset.\label{tu:binary}
\end{align}
\end{subequations}

The following observation about the above formulation is useful computationally. 

\begin{prop}
    Integrality constraints \eqref{tu:binary} can be relaxed.
\end{prop}
\begin{proof}
    It follows straightforward by constraints \eqref{f2:4a} and \eqref{f2:4b} that assure that these variables only take binary values by the integrality of $v$.
\end{proof}

With this observation, the number of binary variables in the problem is $(n + k)|\cset|$. Since the cardinality of $\cset$ is $2^m$, for a fixed (and small) number of hyperplanes the number of binary variables grows linearly with the number of data points and the number of classes. This is much smaller than in the formulation proposed in \cite{BlancoJaponPuerto_ADAC2020}, which requires $n^2 + nm + nk + n$ binary variables. As we  see in our computational experiments, this results in a clear computational advantage.

We next present a derivation of sufficiently large coefficients $\Delta_i$ for use in constraints \eqref{f2:5}-\eqref{f2:6} in the case that the set $\coefset$ has bounded norm.

\begin{prop}
    Assume that $\coefset \subseteq \{ (\hcoef,\hrhs): \| \hcoef \| \leq \maxnorm \}$, where $\| \cdot \|$ is any $p$-norm. Then, for each $\di \in \dset$:
    $$
    \widehat{\Delta}_i = \maxnorm \| x_i \|_* + 2 + \max \{ \maxnorm \| x_i \|_*: i \in \dset\}
    $$
    is a valid value for $\Delta_\di$ in constraints \eqref{f2:5} and \eqref{f2:6}, for all $\di \in \dset$, where $\| \cdot \|_*$ denotes the dual norm of $\| \cdot \|$.
\end{prop}
\begin{proof}
Observe that:
\[ \max\{\hcoef^\top x_i : \| \hcoef \| \leq \maxnorm\} = \max\{ -\hcoef^\top x_i : \| \hcoef \| \leq \maxnorm\}  = \maxnorm \| x_i \|_* .\] 
Next, define $\bar{\Delta}= \max \{ \maxnorm \| x_i \|_*: i \in \dset\}$. Then:
\[1 + |\hcoef^\top x_i | \leq 1 + \bar{\Delta} \] for any feasible $\hcoef$ and data point $x_i$, and hence we can restrict $|\hrhs_\hi| \leq 1 + \bar{\Delta}$. Thus, we derive that
\[ 1 + (\hcoef^\top x_i + b_r) \leq \maxnorm \| x_i \|_* + 2 + \bar{\Delta} =: \Delta_\di \]
and also
\[ 1 - (\hcoef^\top x_i + b_r) \leq \maxnorm \| x_i \|_* + 2 + \bar{\Delta} = \Delta_\di\]
and hence this value can be used for $\Delta_\di$ for each constraint.
\end{proof}

\noindent{\bf Symmetry}

Formulation \eqref{eq:form}
 contains symmetry with respect to permutation of assignment of hyperplanes to splits (and matching re-ordering of cells). Different strategies to mitigate this symmetry can be applied:
\begin{itemize}
    \item A small reduction of symmetry can be accomplished by choosing an arbitrary cell, say $1 \in \cset$, and fixing its class to an arbitrary class, say $1 \in \classes$, i.e., fixing $\cclass_{11} = 1$. We can furthermore pick an arbitrary data point $\di' \in \dset$ having $\dclass_{\di'}=1$ and force this data point to be assigned to this cell, via the constraint $\assign_{\di'1} = 1$. 
\item One can require that the right-hand side of the hyperplanes ($b_r$) be nonnegative in the formulation (one obtain the equivalent hyperplane by changing the signs) and avoid random assignment of \emph{tags} to the hyperplane by imposing that $b_r \geq b_{r+1}$ for all $r = 1,\ldots,m$ and $b_m \geq 0$
\item Some solvers, such as Gurobi or CPLEX, have implemented efficient techniques for symmetry detection, and that can be also useful to avoid extra exploration in the branch-and-bound tree. The amount of effort the solvers spend on their symmetry detection and handling techniques can be controlled with solver settings. For instance, in Gurobi, it may be advantageous to set the parameter \texttt{Symmetry} can be set to aggressive for models that exhibit symmetry.
\end{itemize}

\subsection{Extension to Classification Via Decision Trees}
\label{sec:tree}

The hyperplanes arrangement structure of the problem as described partitions the feature space $\R^d$ by into cells by considering all possible intersections of the half-spaces defined by the selected hyperplanes. A seemingly very different approach is the Optimal Classification Tree \cite{bertsimas2017optimal,aghaei2020learning,blanco2022robust,d2024margin,boutilier2023optimal}, in which the feature space is partitioned by following branches in a tree $\mathcal{T}$, where each branch is defined by a hyperplane. Despite this difference, we discuss in this section how a small modification to formulation \eqref{eq:form} can be used for finding a multivariate classification tree using margin violation as the error metric.

Thus, let $\mathcal{T}$ be a given depth $D$ binary tree with branch nodes $\hset$ and leaf nodes $\cset$. For each branch node $\hi \in \hset$, hyperplane $H_\hi= \{z \in \R^d: \hcoef_\di^{\top} z + \hrhs_\di=0\}$ for each $m \in \hset=\{1, \ldots, m\}$ defines the branch at node $\hi$. For each branch node $\hi$, we define $\cset^+_\hi \subseteq \cset$ as the set of leaf nodes which are descendant from the left branch of node $\hi$, and $\cset^+_\hi \subseteq \cset$ as the set of leaf nodes which are descendant from the right branch of node $\hi$. Note that, different from the pure hyperplanes arrangement structure, with the exception of the root node it holds that $\cset^+_\hi \cup \cset^-_\hi \subsetneq \cset$.

With this modification of the definition of the sets $\cset^+_\hi$ and $\cset^-_{\hi}$ the formulation \eqref{eq:form} assigns a class to each leaf of the tree, assigns each data point to a leaf with matching class, and measures the error of assignments by summing the margin errors of each branch from the root of the tree to the leaf node. 
The error we use in this formulation is similar to margin error used in \cite{d2024margin}, with the significant exception that in our formulation we use the margin error for all branch nodes in the tree, whereas in \cite{d2024margin} the margin error is only used in the branch nodes directly above the leaf nodes.

Extensions to existing optimal classification tree formulations, such as penalizing the number of branches that are used in the tree, can be naturally adapted in this formulation.

\section{A Kernelized Version}\label{sec:3}

Kernel methods play a fundamental role in supervised classification, particularly within the framework of Support Vector Machines (SVMs). By implicitly mapping input data into high-dimensional feature spaces through kernel functions, these methods enable linear classifiers to capture complex, nonlinear decision boundaries in the original input space \cite{scholkopf2002learning,cristianini2000introduction}. This capability is crucial in domains where linear separation is not feasible without transformation. The kernel trick allows such transformations without explicitly computing the high-dimensional embeddings, making classifiers both efficient and powerful in practice. Common kernels, such as the radial basis function (RBF), polynomial, and sigmoid kernels, provide flexibility in adapting to diverse data distributions (see, e.g., \cite{svm}). 

The idea behind kernel methods is to provide a classifier for transformed data through a function $\phi: \R^d \rightarrow \R^{D}$ from the feature space to a (possibly) higher dimensional space. Then, through the use of kernels, generalized products in the shape of $\K(x,x') = \phi(x)^{\top} \phi(x')$, one can obtain the separator. Further details can be found in \cite{svm}.

We discuss how our proposed formulation can be adapted to use a given kernel function $\mathcal{K}: \R^d \times \R^d \rightarrow \R$. The following result provides a formulation for constructing nonlinear separators for Euclidean SVM-based hyperplane coefficients (in the $\phi$-projected space) in the form:
$$
\coefset = \{(\hcoef, \hrhs) \in \R^D: \|\hcoef\|_2 \leq \kappa\}
$$
for a given value of $\kappa>0$. As in our original formulation, the value of $\kappa$ may be considered a hyperparameter of the model.

\begin{theorem}\label{theo:kernel}
    Let $\mathbb{K} = (\K(x_i,x_j))_{i,j \in \dset}$ be the Gram matrix of the kernel function for the training observations and let 
    $\mathbb{K}^\dagger$ be the Moore-Penrose pseudoinverse of  $\mathbb{K}$~(see, e.g., \cite{boyd2004convex}). Then, the following mathematical optimization model allows the construction of kernel-based non linear separators:
    \begin{subequations}
\label{eq:kernel}
    \begin{align}
\min & \;\;\; \sum_{\hi \in \hset} \dsum_{\di\in \dset} \err_{\di\hi} \label{fk:obj}\\
  \mbox{s.t. } & \sum_{\ci \in \cset} \assign_{\di \ci} = 1, && \forall \di \in \dset, \label{fk:1}\\
    & \sum_{\ki \in \classes} \cclass_{\ci \ki} = 1, && \forall \ci \in \cset, \label{fk:2}\\
&  \assign_{\di \ci} \leq \cclass_{\ci \dclass_{\di}}, && \forall \di \in \dset, \ci \in \cset , \label{fk:3}\\
& \sensel_{\hi \di} = \sum_{\ci \in \csetl_{\hi}} \assign_{\di \ci}, && \forall \di \in \dset, \hi \in \hset,  \label{fk:4a}\\
& \senser_{\hi \di} = \sum_{\ci \in \csetr_{\hi}} \assign_{\di \ci}, && \forall \di \in \dset, \hi \in \hset,  \label{fk:4b}\\
& \err_{\di\hi} \geq 1-\lambda_{\di\hi} - \hrhs_\hi - \Delta_\di (1-\sensel_{\hi\di}), &&\forall \di\in \dset,\\
&\err_{\di\hi} \geq 1+\lambda_{\di\hi} + \hrhs_\hi - \Delta_\di (1-\senser_{\hi\di}), &&\forall i\in \dset,\\
& \mathbb{K}\mathbb{K}^\dagger \lambda_\hi = \lambda_\hi, &&\forall \hi \in \hset, \label{eq:lambcond} \\
& \sqrt{\lambda_\hi^\top \mathbb{K}^{\dagger} \lambda_{\hi}} \leq \kappa, &&\forall \hi \in \hset, \label{eq:matnorm} \\
& \lambda_\hi \in \R^n, \hrhs_\hi \in \R, && \forall \hi \in \hset,\\
& \assign_{\di \ci} \in \{0,1\}, && \forall \di \in \dset, \ci \in \cset, \\
& \cclass_{\ci \ki} \in \{0,1\}, && \forall \ci \in \cset, \ki \in \classes, \label{vzk:binary}\\
& \sensel_{\hi\di}, \senser_{\hi\di}, \err_{\di\hi} \geq 0, && \forall \hi\in \hset, \di \in \dset.\label{tuk:binary}
\end{align}
\end{subequations}
\end{theorem}
\begin{proof}
Note that the primal kernelized formulation remains the same, except the definition of the error-based constrains \eqref{f2:5} and \eqref{f2:6} and the hyperplanes coefficient constraint. Thus, defining:
$$
\mathcal{X} = \{(z, v, t, u): \text{ verifying} \eqref{f2:1}, \eqref{f2:2}, \eqref{f2:3}, \eqref{f2:4a}, \eqref{f2:4b} \eqref{vz:binary}, \eqref{tu:binary}\},
$$
the transformed problem problem becomes:
\begin{align}
    \min & \; \sum_{\di\in \dset} \sum_{r\in M} e_{\di\hi}\\
    s.t. &\;  \err_{\di \hi} \geq 1 - (\hcoef_{\hi}^{\top} \phi(x_\di) + \hrhs_{\hi}) - \Delta_{\di} (1-\sensel_{\hi \di}), &\quad& \forall \di \in \dset, \hi \in \hset,  \label{f2:5b}\\
& \; \err_{\di \hi} \geq 1 + (\hcoef_{\hi}^{\top} \phi(x_\di) + \hrhs_{\hi}) - \Delta_{\di}(1- \senser_{\hi \di}), &\quad& \forall \di \in \dset, \hi \in \hset,  \label{f2:6b}\\
& \|\hcoef_\hi\| \leq \kappa, &\quad&\forall \hi \in \hset,\\
& (z, v, t, u) \in \mathcal{X}.
\end{align}

For a fixed vector $(\bar z, \bar v, \bar t, \bar u) \in \mathcal{X}$, the problem is separable by hyperplanes. For each fixed hyperplane $\hi$, the Lagrangean function is:
\begin{align*}
\mathcal{L}(\hcoef,\hrhs,\err;\alpha,\beta,\gamma^+, \gamma^-, \rho) \;&=\;
\sum_{\di\in \dset} \err_\di \\
&\quad + \sum_{\di\in \dset} \alpha_\di\Big(1 - \Delta_\di(1-\bar \sensel_{\hi\di}) - \hcoef_r^\top \phi(x_\di) - \hrhs_\hi - \err_i\Big) \\
&\quad + \sum_{\di\in \dset} \beta_\di\Big(1 - \Delta_\di(1-\bar \senser_{\hi\di}) + \hcoef_r^\top \phi(x_\di) + \hrhs_\hi - \err_i\Big) \\
&\quad - \sum_{\di\in \dset} \gamma_i^+ \err_i \\
&\quad - \sum_{\di\in \dset} \gamma_i^- \err_i \\
&\quad + \rho \big(\|\hcoef_\hi\|_2^2 - \kappa^2\big).
\end{align*}
The KKT conditions read:
\begin{align*}
&\quad \sum_{\di\in \dset} (\alpha_\di - \beta_\di)\phi(x_\di) = 2 \rho \hcoef_\hi & \Rightarrow & \;\;\hcoef_\hi \in {\rm span}(\phi(x_1), \ldots, \phi(x_n)), \\
&\quad \sum_{\di\in \dset} (\alpha_\di - \beta_\di) = 0 & \Rightarrow & \sum_{\di\in \dset} \alpha_\di = \sum_{\di\in \dset}\beta_\di,  \\
&\quad \alpha_\di + \gamma^+_\di = 1, \;  \di\in \dset & \Rightarrow & 0 \leq \alpha_\di \leq 1, \\[1ex]
&\quad \beta_\di + \gamma^-_\di = 1,\; \di\in \dset & \Rightarrow & 0 \leq \beta_\di \leq 1.
\end{align*}
In particular,  for the optimal solution of the problem we get that $\hcoef_\hi = \sum_{j \in \dset} q_{j\hi} \phi(x_j)$ for some coefficients $q_\hi = (q_{1\hi}, \ldots, q_{n\hi}) \in \R^n$, for all $\hi \in \hset$. Then, defining $\lambda_{\hi\di} = \hcoef_\hi^\top \phi(x_\di)$:
$$
\lambda_{\hi\di} = \hcoef_\hi^\top \phi(x_\di) = \sum_{j \in \dset} q_{\di\hi} \phi(x_j)^t \phi(x_\di) = \sum_{j \in \dset} q_{\di\hi} \mathbb{K}_{ij} \Rightarrow \lambda_\hi = \mathbb{K} q_r.
$$
Thus, the existence of $\lambda_\hi$ is assured by the condition $\mathbb{K}\mathbb{K}^\dagger \lambda_r = \lambda_r$, for all $\hi \in \hset$, .

Furthermore, $\|a_r\|_2^2 = \displaystyle\min_{q_r: \mathbb{K} q_r = \lambda_\hi} q_r^{\top} \mathbb{K} q_r = \lambda_\hi^{\top} \mathbb{K}^\dagger \lambda_\hi
$

Summarizing, the transformed problem is equivalent to:
\begin{align*}
\min & \;\;\; \sum_{\hi \in \hset} \dsum_{\di\in \dset} \err_{\di\hi} \\
\mbox{s.t. }  
& \err_{\di\hi} \geq 1-\lambda_{\di\hi} - \hrhs_\hi - \Delta_\di (1-\bar\sensel_{\hi\di}), \forall \di\in \dset,\\
&\err_{\di\hi} \geq 1+\lambda_{\di\hi} + \hrhs_\hi - \Delta_\di (1-\bar\senser_{\hi\di}), \forall i\in \dset,\\
& \mathbb{K}\mathbb{K}^\dagger \lambda_\hi = \lambda_\hi, \forall \hi \in \hset,\\
& \sqrt{ \lambda_\hi^\top \mathbb{K}^{\dagger}\lambda_\hi } \leq \kappa, \forall \hi \in \hset,\\
& \lambda_\hi \in \R^n, \hrhs_\hi \in \R, \forall \hi \in \hset,\\
& \err_{\di\hi} \geq 0, \forall \di \in \dset, \hi \in \hset.
\end{align*}

The desired result is obtained by integrating the problem above with the decision made in the set $\mathcal{X}$, that allows also to decide the values of the variables $z$, $v$, $\senser$, and $\sensel$.
\end{proof}

Observe that in the special case that the gram matrix $\mathbb{K}$ is invertible, so that $\mathbb{K}^\dagger = \mathbb{K}^{-1}$, the constraint \eqref{eq:lambcond} becomes redundant, and the constraint \eqref{eq:matnorm} can be written as 
\[ \| \lambda_r \|_{\mathbb{K}^{-1}} \leq \kappa \]
where $\| \lambda \|_{A} = \sqrt{\lambda^\top A\lambda}$ is the $A$-norm defined by a positive definite matrix $A$.

Model \eqref{eq:kernel} allows to find the classifier for the training observations. In order to predict the class with the obtained kernelized classifier to out-of-sample, $x \in \R^d$, observation one has to proceed as detailed in the following result.

\begin{prop}
    Let $(\bar \lambda, \bar b, \bar z)$ a (partial) optimal solution of \eqref{eq:kernel}. Then, for each $x \in \R^d$, the predicted class of $x$ with this classifier is $\widehat{y}_x = \sum_{k \in \classes} k \bar z_{c^x}k$, where $c^x \in \cset$ is defined as:
    $$
    c_r^x = \sign\Big((\mathbb{K}^\dagger \bar \lambda_\hi)^{\top} \big(\K(x_1, x), \ldots, \K(x_n, x)\big)  + \bar\hrhs_\hi\Big), \quad \forall \hi \in \hset.
    $$
\end{prop}
\begin{proof}
    Note that, the predicted class of $x$ is obtained first by detecting the cell (with respect to the obtained \emph{hyperplanes}) where it is positioned, and the applying the assignment $\bar z$ to that cell. The evaluation of $x$ in the (transformed hyperplane) would be (using the rewriting detailed in the proof of Theorem \ref{theo:kernel})
    $$
    \bar \hcoef^{\top}_\hi \phi(x) + \bar \hrhs_\hi = q_\hi^{\top} \mathcal{K}_x + \hrhs_\hi =(\mathbb{K}^\dagger \bar \lambda_\hi)^{\top} \mathcal{K}_x + \bar \hrhs_\hi, \forall \hi \in \hset,
    $$
    where we denote by $\mathcal{K}_x = \big(\K(x_1, x), \ldots, \K(x_n, x)\big)$.

    Then, the evaluation of the branch (left or right) with respect to each of the transformed hyperplanes can be obtained using the values of $\bar \lambda$ and $\bar \hrhs$. The result follows just by applying the assignment $\bar z$ to the obtained sign-pattern.
\end{proof}

In the following example we illustrate the usage of the kernelized methodology.

\begin{example}\label{ex:2}
    We consider the instance of Example \ref{ex:1} but apply the non linear mapping to the data $\phi(x_1, x_2) =  (x_1,\, x_2 +  \sin(0.7 x_1))$. Applying our methodology with a rbf kernel ($\mathcal{K}(x,y)= \exp(-\sigma \|x-y\|^2)$) with $\kappa=10$, $\sigma=1$, and two separating surfaces ($m=2$) we obtained the classifier whose boundaries are plot in  Figure \ref{ex:phi}.
    \begin{figure}[h]
    \begin{center}
\includegraphics[width=0.6\textwidth]{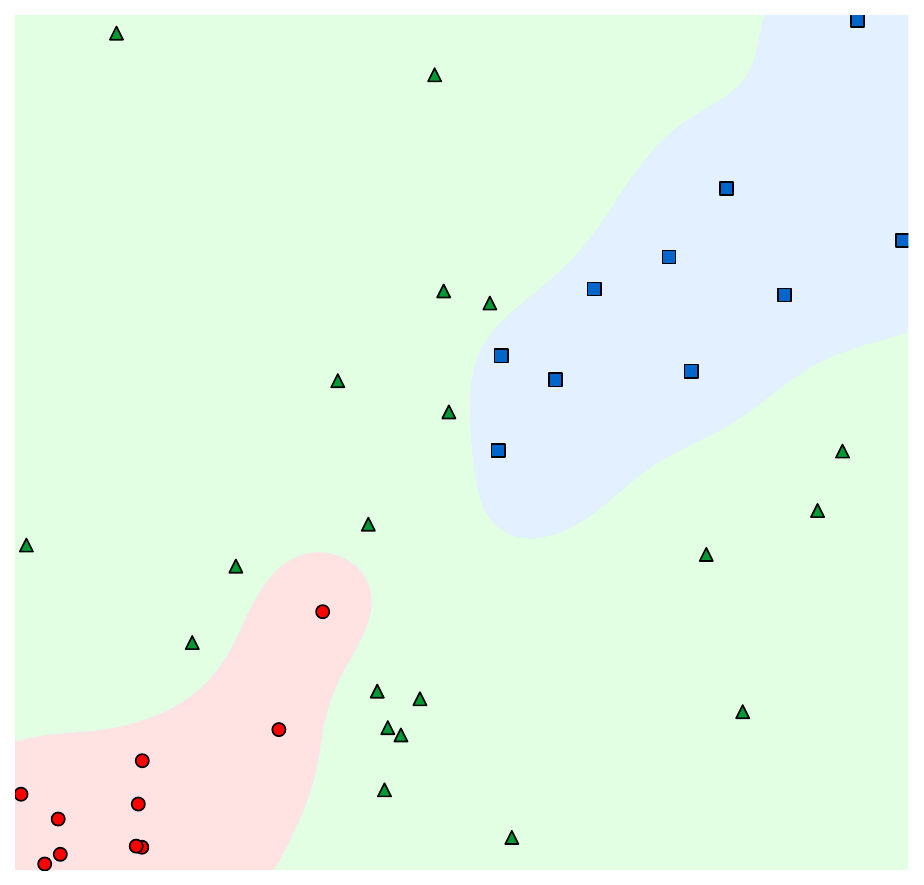}
\end{center}
    \caption{Solution obtained with the kernelized model for Example \ref{ex:2}}
    \label{ex:phi}
    \end{figure}
\end{example}

\exclude{
\begin{enumerate}
    \item[ST1] Splits the classes in two (random) sets $C_1, C_2 \subset \{1, \ldots, k\}$. Set $N_1$ and $N_2$ the observations with classes at each of the class sets, and randomly select subsamples of both sets (size $\sim n/4$). Compute the SVM classifier and predict the classes of each of the observations. For those predicted groups ($T$ and $U$) add cut induced by $\Phi(T,U)$.
    \item[ST2] Select a random subsample of $N$, and separate it by classes: $N_1, \ldots, N_k$. For each pair of classes  add the cut induced by $\Phi(N_{k_1}, N_{k_2})$.
    \item[ST3] For each class $k$, let $N_k$ the observations in that class. Solve $\Phi(N_k, N')$ with $N'$ the observations in a different class that $k$, but close to the observations in $N_k$ (${\rm d}(N_k, x_i) <0.5$ for all $i\in \dset'$ (normalized data). Add the induced cut. Add also cuts induced by small subsamples of $N_k$ and $N'$.
    \item[ST4] Add cuts induced by partitions $T$ and $U$ with $T$ being a singleton and $U$ the set of elements in a different class than the unique element in $T$.
    \item[ST5] Select subsamples of each of the classes of size $d$ (dimension of the space): $X$. Compute generators of the null space of $[X \;\; \mathbf{1}]$, $(a,b) \in \R^{d+1}$ (hyperplane passing through the elements in $X$). Add the cuts induced by the separation of that hyperplane.
    \end{enumerate}
}

\section{Computational Experiments}\label{sec:5}

In this section we provide the results of our computational experiments to to investigate the computational efficiency and predictive accuracy of our proposed model.

All computational experiments were performed on a machine equipped with an Intel(R) Xeon(R) Gold 6258R CPU running at 2.70GHz, supporting the SSE2, AVX, AVX2, and AVX512 instruction sets. The system has 56 physical cores (56 logical processors), and up to 32 threads were utilized during execution. Optimization tasks were solved using Gurobi Optimizer version 10.0.3. with a work limit of 3600 seconds.

\subsection{Synthetic Datasets}

We generate a battery of datasets constructed to simulate a multiclass classification scenario in which each class is represented by multiple, spatially distinct clusters. First, a set of $n \in  \{10, 20, 30, 40, 50, 100, 200, 500, 1000\}$ points is generated using the \texttt{make\_blobs} function of sklearn, creating $n \times b$ clusters in a two-dimensional space with a specified spread, for $b \in \{2,3\}$. These clusters are then randomly reassigned to $k \in \{2,3\}$ classes such that each class corresponds to several disjoint clusters. To increase dimensionality, additional noise features are added: each is drawn independently from a standard normal distribution, so the final dataset has $d \in \{2,5,10\}$ dimensions, with only the first two being informative. Finally, the entire dataset is scaled to the $[0,1]$ range using min-max normalization. This results in a high-dimensional dataset where each class forms several well-separated clusters in feature space, making it suitable for evaluating clustering and classification algorithms under realistic complexity. For each combination of parameters, $(n,b,k,d)$ we generate $5$ random instances. In Figure \ref{fig:example} we show one of our generated instances for parameters $(1000, 3, 3, 10)$, but projected onto the informative dimensions.

\begin{figure}[h!]
\centering\input{plot.tex}
\caption{Illustration of one of our 10-dimensional instances (projected onto the two first dimensions) with 3 classes with three clusters for each of them for a sample to size $1000$.\label{fig:example}}
\end{figure}
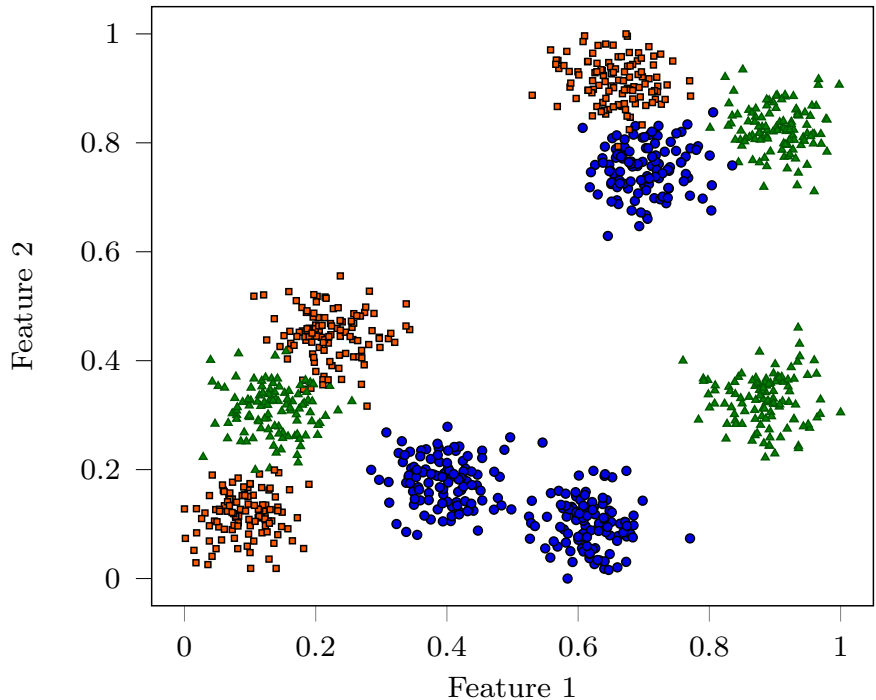

\subsection{Computational Performance}

We first investigate the computational performance of our proposed MIP model. We investigate the different symmetry breaking strategies and then compare the performance of our proposed MIP model to the prior MIP model of \cite{BlancoJaponPuerto_ADAC2020}.

Since the MIP models might be very time demanding, in these first experiments, we only run the models for instances with sample sizes $n \leq 100$. For each of these 356 instances, we run our models to construct $m \in \{1,2,3,4\}$ hyperplanes with $k \leq 2^m$ (to assure feasibility) instances, and use margin value $\kappa=2m$.

\subsubsection*{Symmetry Breaking Strategies}

As discussed in Section \ref{sec:2}, the MIP formulation that we propose exhibits symmetry of different sources that can affect the performance of the solution approaches implemented in the solvers. We explore different strategies to reduce those symmetries, and report results comparing the performance of them.

We denote by \texttt{sym1}, Gurobi's internal symmetry-handling capabilities in an aggressive mode (\texttt{model.Params.Symmetry} = 2). The second method, \texttt{sym2}, enforces an ordering on the bias terms by adding constraints $b_r \geq b_{r+1}$ for all $r < m$ and requires nonnegativity $b_r \geq 0$ for all $r$. The third strategy, \texttt{sym3}, fixes the assignment of the first point to the first cell by imposing $v_{11} = 1$.

In Figure \ref{fig:sym} we plot the performance profile for the different approaches. This plot displays the number of instances solved within the CPU time displayed on the $x$-axis. 
For those instances that are solved in less CPU time, it seems that strategy \texttt{sym3} is considerable better than the others. However, for the most challenging instance, all the methods seems to have a similar performance. 

\begin{figure}[h!]
\centering\includegraphics[width=0.7\textwidth]{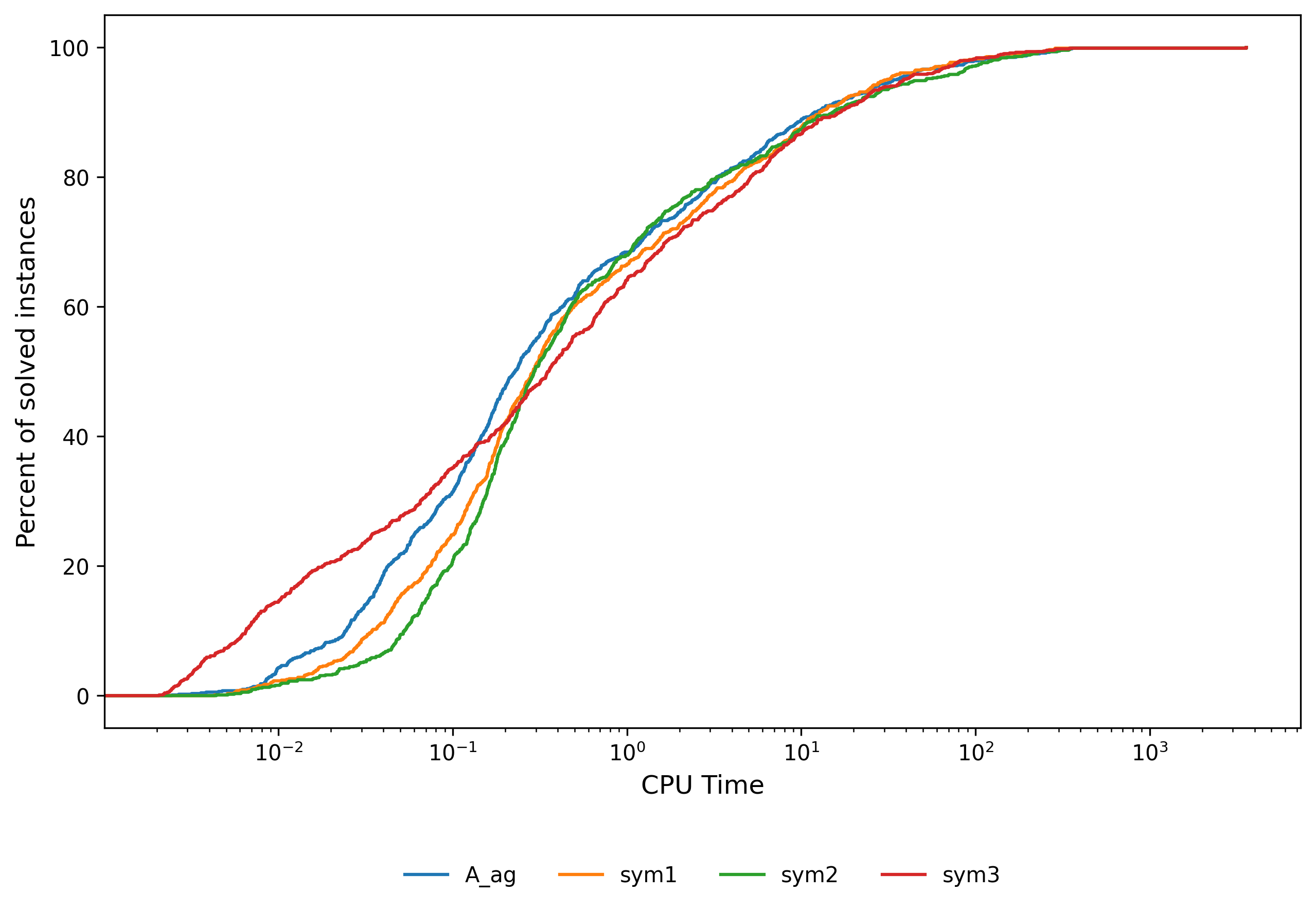}
\caption{Performance Profile by size of the datasets by the breaking symmetry strategy used. CPU time is in log-scale.\label{fig:sym}}
\end{figure}



\subsubsection*{Comparison with Prior MIP Formulation}

We next compare the computational performance of our model with the model proposed in \cite{BlancoJaponPuerto_ADAC2020}, that we denote by \texttt{BJP20}, which also consists on solving a MIP problem to adequately locate a number of hyperplanes for multiclass SVM-based classification (see Section \ref{sec:2}). In Figure \ref{fig:bjp} we plot the performance profile comparing our aggregated model with the symmetry breaking strategy \texttt{sym3} with \texttt{BJP20}. There, one can clearly observe, that our proposed formulation demonstrates substantially superior computational efficiency compared to \texttt{BJP20}, achieving markedly lower CPU times across all tested instances. For example, nearly 90\% of these instances are solved by our formulation within 10 seconds, whereas the model from \texttt{BJP20} fails to solve any of the instances in that time.

\begin{figure}[h!]
\centering\includegraphics[width=0.7\textwidth]{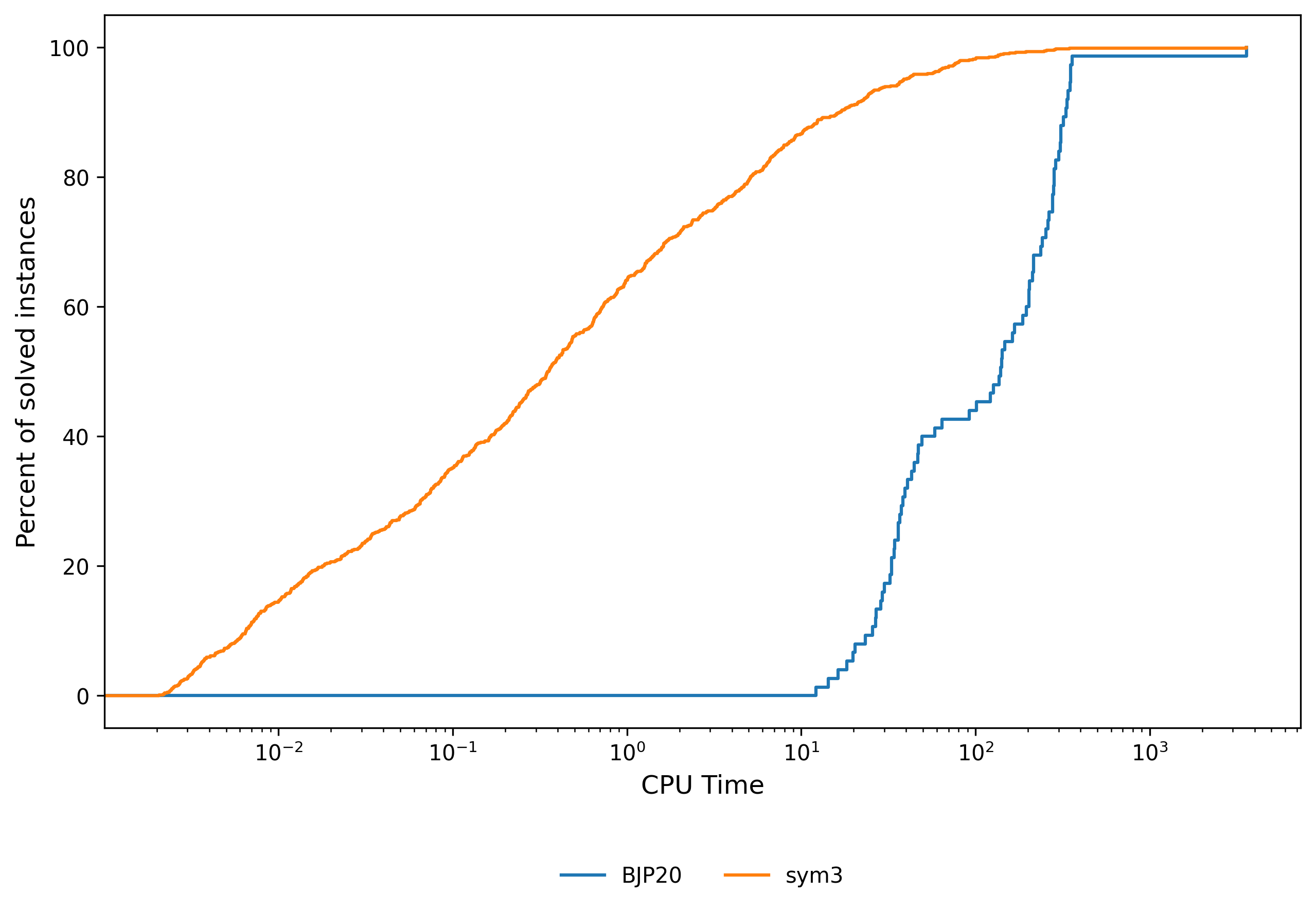}
\caption{Performance profile displaying number of instances solved within the CPU time displayed on the $x$-axis for our formulation with symmetry strategy \texttt{sym3} and \texttt{BJP20}. CPU time is in log-scale.\label{fig:bjp}}
\end{figure}

We present in Table \ref{tab:comparisons} more detailed view of these computational results. There, for the two methods we report the average CPU times (in seconds) required to solve the instances among those that were optimally solved within the time limit (here \texttt{TL} indicates that none of the instances were optimally solved), the average MIP Gaps (in percent) obtained for those instances that were not optimally solved in that limit, and the percent of instances that were not optimally solved. These results are summarized by the number of data points $n$ and the number of hyperplanes $m$. 

The results clearly show the substantial computational improvement achieved by the our approach. For all tested instance sizes, our model is able to solve nearly all cases to optimality within a few seconds, whereas the approach proposed in \cite{BlancoJaponPuerto_ADAC2020} rapidly becomes intractable even for moderate sample sizes. The average CPU times of our method are several orders of magnitude smaller, and the percentage of unsolved instances is negligible compared with the baseline. In addition, the MIP gaps reported for our formulation remain small when the time limit is reached, confirming its numerical stability and tighter continuous relaxation. These results justify the reformulation effort and demonstrate that the proposed model significantly enhances tractability without compromising solution quality or modeling flexibility.

\begin{table}[h]
\centering
\begin{tabular}{l l r r r r r r}
 &  & \multicolumn{2}{c}{CPU Time} & \multicolumn{2}{c}{\% MIP Gap} & \multicolumn{2}{c}{\% UnSolved} \\
$n$ & $m$ & Ours & \cite{BlancoJaponPuerto_ADAC2020} & Ours & \cite{BlancoJaponPuerto_ADAC2020} & Ours & \cite{BlancoJaponPuerto_ADAC2020} \\
\hline
10 & 1 & 0.1 & 31.8  & - & - & 0.0\% & 0.0\% \\
   & 2 & 0.2 & 244.1 & - & 13.5\% & 0.0\% & 60.0\% \\
   & 3 & 2.1 & 324.8 & - & 45.7\% & 0.0\% & 91.7\% \\
   & 4 & 4.3 & TL       & 50.0\% & 84.0\% & 3.3\% & 100.0\% \\ \hline
20 & 1 & 0.1 & 145.3 & - & 39.9\% & 0.0\% & 73.3\% \\
 & 2 & 0.8 & 277.7 & - & 96.0\% & 0.0\% & 98.3\% \\
 & 3 & 14.6 &TL & - & 100.0\% & 8.3\% & 100.0\% \\
 & 4 & 1.4 &TL & 99.0\% & 100.0\% & 16.7\% & 100.0\% \\\hline
30 & 1 & 0.1 & 216.4 & - & 56.8\% & 0.0\% & 83.3\% \\
 & 2 & 2.3 &TL & - & 100.0\% & 3.3\% & 100.0\% \\
 & 3 & 12.6 &TL & 49.3\% & 100.0\% & 26.7\% & 100.0\% \\
 & 4 & 4.8 &TL & 93.9\% & 100.0\% & 25.0\% & 100.0\% \\\hline
40 & 1 & 0.1 & 147.4 & - & 81.2\% & 0.0\% & 96.7\% \\
 & 2 & 4.6 &TL & 55.9\% & 100.0\% & 15.0\% & 100.0\% \\
 & 3 & 24.2 &TL & 69.5\% & 100.0\% & 45.0\% & 100.0\% \\
 & 4 & 17.2 &TL & 92.1\% & 100.0\% & 35.0\% & 100.0\% \\\hline
50 & 1 & 0.1 &TL & - & 98.8\% & 0.0\% & 100.0\% \\
 & 2 & 13.8 &TL & 73.1\% & 100.0\% & 15.0\% & 100.0\% \\
 & 3 & 24.7 &TL & 78.6\% & 100.0\% & 55.0\% & 100.0\% \\
 & 4 & 25.9 &TL & 90.6\% & 100.0\% & 56.7\% & 100.0\% \\\hline
100 & 1 & 0.1 &TL &- & 100.0\% & 0.0\% & 100.0\% \\
 & 2 & 49.0 &TL & 60.4\% & 100.0\% & 50.0\% & 100.0\% \\
 & 3 & 2.6 &TL & 97.5\% & 100.0\% & 83.9\% & 100.0\% \\
 & 4 & 5.6 &TL & 98.9\% & 100.0\% & 83.9\% & 100.0\% \\\hline
\end{tabular}
\caption{Average CPU times (in seconds) and MIP Gaps obtained in our computational experiments when comparing our method with the one proposed in \cite{BlancoJaponPuerto_ADAC2020}.\label{tab:comparisons}}
\end{table}

\subsection{Prediction Accuracy}

We next compare the performance, in terms of training accuracy, of our method with respect to three well known SVM-based methodologies: One Versus One (ovo) strategy, which constructs $k(k - 1)$ hyperplanes and was tested in both its linear (\texttt{sk\_OVOlin}) and rbf kernel (\texttt{sk\_OVOrbf})) kernel variants; the One Versus Rest (ovr) strategy, which uses $k$ hyperplanes, also in both linear (\texttt{sk\_OVRlin}) and rbf kernel (\texttt{sk\_OVRrbf}) versions; and the Crammer-Singer (\texttt{cs}) method, as proposed in \cite{cs}. All these methods are implemented in module \texttt{sklearn} in python. We use the default parameters for all of them. We conduct this test only for the instances with sample size $n \geq 100$.

In terms of computational performance, our MIP-based approach is computationally more costly than the methods in \texttt{sklearn}, which generally take less than one second on these data sets. However, we explore here the ability of our methodology to capture class geometries that other methodologies are not able to detect.

In Figure \ref{fig:bp} we show the box plots for the  training accuracies (percent of well classified instances) with all these methods. As can be observed, \texttt{BJP20} achieves high accuracies in some datasets, but has poor accuracy in many others. This is due to the computational challenge in solving this model, since $94\%$ of the instances were not optimally solved within the time limit. OVO and OVR  (linear) seems to fail for these instances, obtaining small average accuracies close to $75\%$ compared to our methodology ($94\%$). The only method that has comparable training accuracy as ours is OVR with rbf kernel ($93\%$). However, the model obtained from our methodology uses linear separating hyperplanes, and hence may be considered as more interpretable and accessible since it does not require the use of kernels to obtain accurate boundaries.

It is worth emphasizing that the ability to solve the proposed formulation to optimality plays a key role in the quality and reliability of the resulting classifiers. Since our formulation (and the strengthening strategies that we apply) reduces the number of binary variables and symmetry in the model (with respect to previous discrete optimization-based approaches), it allows optimal solutions to be obtained for significantly larger instances within much shorter CPU times. This, in turn, ensures that the resulting separating hyperplanes fully exploit the margin-based structure of the model, providing more accurate and stable classifiers. In contrast, the previous formulation \cite{BlancoJaponPuerto_ADAC2020} often remains far from optimality due to its high combinatorial complexity, which explains its lower accuracy in many instances. The improved tractability of our model therefore translates directly into better predictive performance, while substantially reducing the total computational effort required during training.

\begin{figure}[h!]
\centering\includegraphics[width=0.8\textwidth]{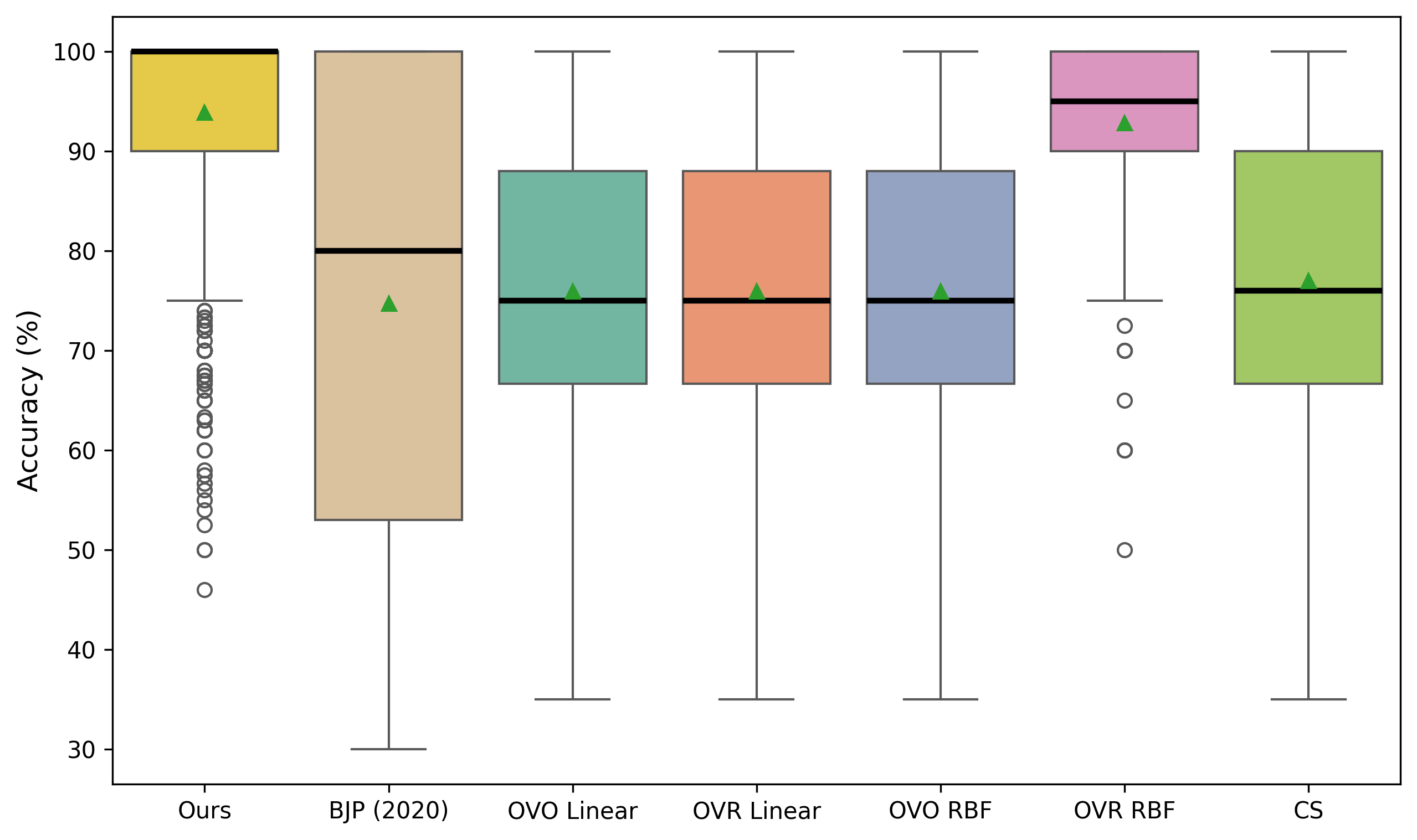}
\caption{Boxplot with accuracies for the different methods with default parameters. For each method, the thick black horizontal line represents the median and the triangle represents the average. \label{fig:bp}}
\end{figure}

Although our experimental analysis primarily focuses on comparisons with the most widely used SVM-based multiclass schemes, the proposed framework could also encompass and be contrasted with other optimization-based classifiers, such as optimal classification trees or mixed-integer models with embedded feature selection \cite{bertsimas2017optimal, aghaei2020learning, labbe2019mixed}. For the sake of simplicity and concreteness, these extensions are not explored here, as their analysis would require additional methodological developments beyond the scope of this work. The aim of this paper is to establish a general and unified optimization-based framework for multiclass classification, providing the theoretical and computational foundations upon which such comparative studies can later be built.

\subsection{Scaling Up Via Dynamic Clustering}
Although our proposed MIP formulation is significantly more tractable than the prior formulation \texttt{BJP20}, exactly solving the MIP formulation to optimality may still be intractable for instances with a large number of data points. Thus, following the motivation of the classic ``chunking'' methods for training SVM models \cite{vapnik1982estimation}, or the recently proposed dynamic clustering method used for MIP approaches to semi-supervised SVMs  \cite{burgard2024mixed}, we propose a dynamic clustering strategy that repeatedly solves the proposed MIP formulation on a reduced size-data set obtained by creating representative data points.

The approach, described in Algorithm \ref{alg:dyncluster} is designed to refine an initial partition of a dataset $\mathcal{X}$ into clusters using a model-driven splitting strategy. The procedure begins by initializing a set of clusters $\mathcal{C}$ such that the number of clusters within each class is $n_{\text{clusters}}$, using Kmeans applied to the feature vectors sharing each class. Thus, we construct the clustered dataset,  $\mathcal{X}_{\mathcal{C}}$, with the centroids of these clusters and their associated label. 

At each iteration, the model is trained using the current centroids $\mathcal{X}_{\mathcal{C}}$, yielding a candidate solution for the arrangement of hyperplanes. Then, the total classification error over the original dataset $\mathcal{X}$ is evaluated, denoted $e_{\mathcal{X}}$. If no significant improvement is observed in the error value from the most recent iteration to the current iteration, the procedure terminates.
Otherwise, the current clustering $\mathcal{C}$ is dynamically updated by splitting clusters whose members are separated by the decision boundaries of the classifier.  This results in a new cluster assignment $\mathcal{C}'$ and its corresponding set of centroids $\mathcal{X}_{\mathcal{C}'}$.
If the new centroids do not differ from the previous ones, then the algorithm stops. Otherwise, the updated clusters and centroids are adopted and the process continues. In addition to the two mentioned stopping rules (lack of improvement in errors or lack of change in clusters), we also impose a time limit.

This algorithm balances model learning with clustering structure discovery and is particularly effective for our multiclass classification framework where cluster geometry is captured via hyperplane separability. The approach returns a heuristic solution based on the final clustering configuration.

\begin{algorithm}[h!]
\caption{Dynamic clustering approach}\label{alg:dyncluster}
\KwIn{$\X$, $n_{\text{clusters}}$}
\KwOut{Heuristic Solution}

Initialize clusters $\mathcal{C}$ ($|\mathcal{C}|=n_{\text{clusters}}$) and labels, and compute the centroid dataset $\X_{\mathcal{C}}$.

\While{time elapsed $\leq$ time limit}{
    Train model using $\X_{\mathcal{C}}$
    $e_\X \gets$: error over all data\;

    \If{No improvement}{
            \textbf{break}\;
        }

    Update clusters: Split clusters if separated by hyperplane and merge small clusters: $\mathcal{C}'$ and $\X_\mathcal{C'}$.

    \eIf{$X_{\mathcal{C}'} = \X_{\mathcal{C}}$}{
        \textbf{break} 
    }{
    $X_{\mathcal{C}} \gets X_{\mathcal{C}'}$.
    }
    }

\Return Solution\;
\end{algorithm}

We next investigate the computational performance and quality of the solutions obtained using the dynamic clustering approach. To this end, we run, for one hour of time limit our exact model, and the clustering approach for $m \{2,3\}$ hyperplanes, $\kappa = 3m$, and a maximum number of $100$ iterations for the clustering heuristic.

In Figure \ref{fig:time_clusters} we show the CPU time required by both approaches (the so-called \texttt{plain} and \texttt{cluster}) to terminate (unless the time limit or maximum number of iterations is reached) by number of observations. We observe that, as expected, the CPU times required to solve the instances with the dynamic clustering approach are much smaller than those of the exact approach, especially as the number of data points grows.

\begin{figure}[h]
\centering
\includegraphics[width=0.8\textwidth]{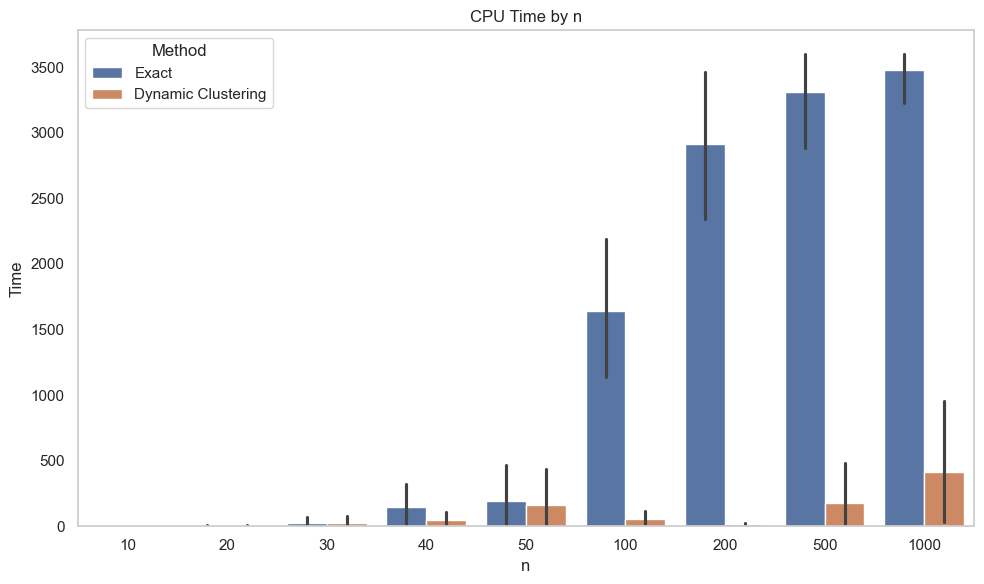}
\caption{CPU times of both the exact and the dynamic clustering. \label{fig:time_clusters}}
\end{figure}

We next investigate the quality of the solutions obtained by the dynamic clustering approach. In Figure \ref{fig:devs} we show two quality measures for the solutions obtained with the clustering approach. On the one hand, we show a box plot indicating the percent deviation in the evaluation of the actual objective function of the solutions obtained with the clustering method with respect to the exact approach (left plot). Second, we measure the training accuracy of the  solutions obtained with both methods and present a box plot of this deviation.  From both plots, we can conclude that, except for a few instances, the clustering algorithm exhibits excellent performance in terms of quality. Its deviation in objective value and accuracy is below $5\%$, and in many cases it even produces solutions superior to those the exact approach obtains within the given time limit.

\begin{figure}[h]
\centering
\includegraphics[width=0.45\textwidth]{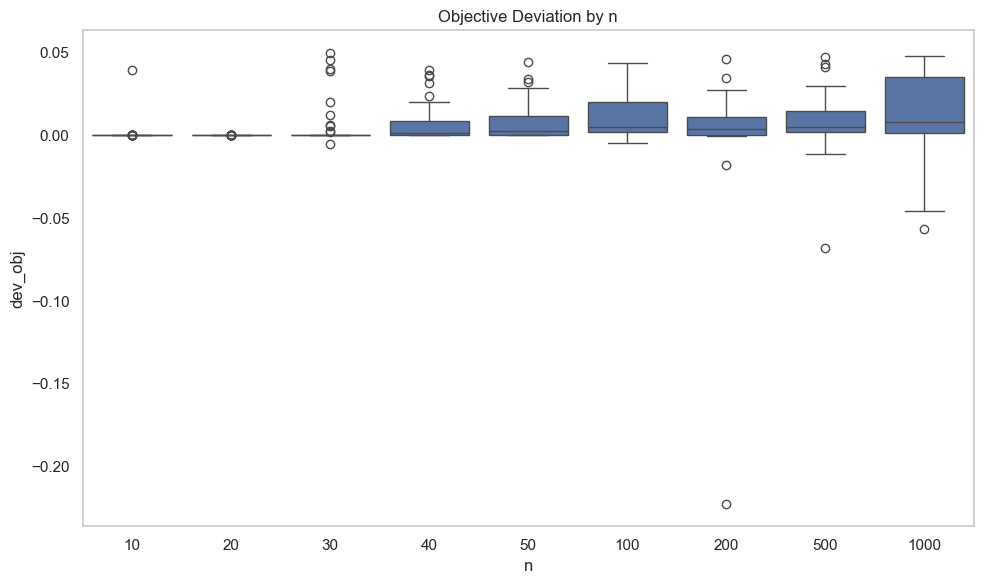}~\includegraphics[width=0.45\textwidth]{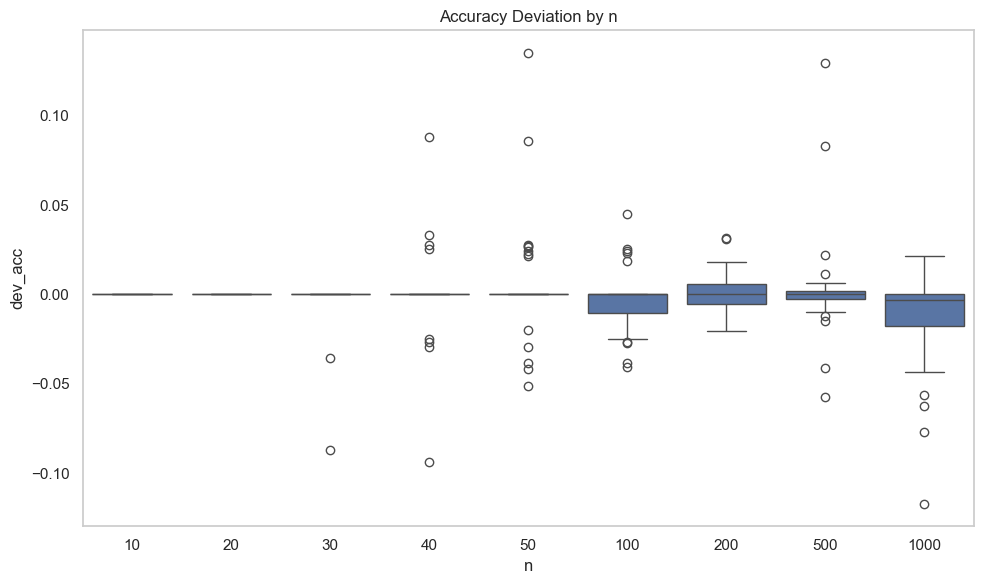}
\caption{Deviations of the solutions obtained with the dynamic clustering with respect the exact approach (left: deviation in the objective function; right: deviation in training accuracy). Negative deviation indicates the dynamic clustering method produced a better value than the solution obtained from solving the original formulation with a one hour time limit. \label{fig:devs}}
\end{figure}

While the proposed dynamic clustering matheuristic significantly improves scalability relative to the exact MIP, very large-scale datasets (e.g., with tens of thousands of observations or hundreds of features) may still pose computational challenges. In such contexts, hybrid strategies that combine our formulation with decomposition, incremental training, or parallelization techniques could provide additional efficiency gains and will be the subject of future work.

\subsection{Performance on Real Datasets}

In our final experiment, we test the scalability of using our MIP formulation within the dynamic clustering approach by testing it on real datasets from the UCI repository.  We run our method and those available in sklearn (OVR, OVO, and Cramer-Singer) with different kernels (linear and rbf). We tune the parameters running the different parameters on a dense grid $C$, $\sigma \in [0.01, 100]$ (30 values), for OVR, OVO and CS, and $\kappa \in [0.1, 50]$ (20 values) and $m \in \{m_0, m_0+1, m_0+2\}$ for our approach, where $m_0$ is the minimum number of hyperperplanes required to have at least the same number of cells than classes in the dataset).

In Table \ref{tab:uci} we show the datasets used in our experiments, as well as the  best  test accuracy results results obtained with the four approaches for each dataset.

\begin{table}[t]
\centering
\setlength{\tabcolsep}{4.5pt}      
\renewcommand{\arraystretch}{0.95}  
\resizebox{\textwidth}{!}{%
\begin{tabular}{l r r r c c c c c c c}
\toprule
\textbf{dataset} & $n$ & $d$ & $k$ & CS & OVO (L) & OVO (K) & OVR (L) & OVR (K) & Ours (L) & Ours (K) \\
\midrule
\texttt{appendicitis\_test} & 106 & 7 & 2 & \textbf{0.9245} & \textbf{0.9245} & \textbf{0.9245} & \textbf{0.9245} & \textbf{0.9245} & \textbf{0.9245} &  \textbf{0.9245}\\
\texttt{heart-statlog-uci} & 270 & 13 & 2 & 0.8593 & 0.8667 & 0.8370 & 0.8667 & 0.8519 & 0.8148 & \textbf{0.8767} \\
\texttt{iris} & 150 & 4 & 3 & 0.9600 & 0.9733 & 0.9733 & 0.9333 & 0.9600 & 0.9733 & \textbf{0.9866} \\
\texttt{letter} & 3864 & 16 & 5 & 0.9043 & 0.9200 & 0.9352 & 0.8973 & 0.9352 & 0.9352 &  \textbf{0.9357}\\
\texttt{maternal\_health\_risk} & 1014 & 6 & 3 & 0.6218 & 0.6036 & 0.6681 & 0.5812 & 0.6737 & 0.6232 & \textbf{0.6918} \\
\texttt{optdigits} & 2822 & 64 & 5 & 0.9853 & 0.9881 & 0.9786 & 0.9869 & \textbf{0.9941} & 0.9694 &  0.9881\\
\texttt{page-blocks} & 5473 & 10 & 5 & 0.9401 & 0.9329 & 0.9424 & 0.9310 & 0.9521 & 0.9408 & \textbf{0.9526} \\
\texttt{phoneme} & 5404 & 5 & 2 & 0.7755 & 0.7755 & 0.7835 & 0.7755 & \textbf{0.8266} & 0.7835 & \textbf{0.8266} \\
\texttt{satimage} & 4922 & 36 & 5 & 0.8944 & 0.9026 & 0.9128 & 0.8884 & 0.9102 & 0.8994 & \textbf{0.9149} \\
\texttt{Thoracic\_Surgery} & 470 & 3 & 2 & \textbf{0.8596} & \textbf{0.8596} & \textbf{0.8596} & \textbf{0.8596} & \textbf{0.8596} & \textbf{0.8596} & \textbf{0.8596} \\
\texttt{vehicle} & 846 & 18 & 4 & 0.7729 & 0.7912 & 0.8205 & 0.7766 & 0.8223 & 0.7099 & \textbf{0.8351} \\
\texttt{wine} & 178 & 13 & 3 & \textbf{1.0000} & 0.9775 & 0.9775 & \textbf{1.0000} & 0.9775 & 0.9887 & \textbf{1.0000} \\
\texttt{yeast} & 598 & 8 & 5 & 0.8863 & 0.8896 & 0.8896 & 0.8729 & 0.8963 & 0.9096 & \textbf{0.9163} \\
\bottomrule
\end{tabular}%
}
\caption{Accuracy results for the UCI datasets. \label{tab:uci}}
\end{table}

Across the collection of multiclass datasets, our proposed methods, based respectively on constructing arrangements of \emph{linear} hyperplanes (Ours (L)) and their nonlinear kernel extension (Ours (K)), exhibit competitive and often superior accuracy compared to classical decomposition approaches such as one-vs-one (OVO) and one-vs-rest (OVR), both in their linear and kernelized variants. In several datasets (e.g., heart-statlog-uci, iris, page-blocks, phoneme, satimage, student\_lifestyle\_dataset, vehicle, wine, and yeast), at least one of our methods achieves the best overall performance, with the kernelized version frequently providing an additional improvement when nonlinear decision boundaries are relevant. Notably, the linear version already matches or surpasses standard methods on high-dimensional or moderately sized problems such as letter, optidigits, and maternal\_health\_risk, suggesting that the cell-based hyperplane arrangement effectively captures class separability even without kernel transformations. The kernel-based variant further refines this separation, attaining near-perfect accuracy on benchmark datasets such as iris and wine. Overall, these results confirm that our framework offers a flexible and robust alternative to traditional multiclass strategies, particularly effective in capturing both linear and nonlinear class structures.

\section{Conclusion}

We have proposed a new MIP formulation for finding a structured hyperplane arrangement for multiclass classification. The formulation naturally extends the SVM margin modeling approach to a setting with multiple hyperplanes and classes. The proposed formulation is computationally more efficient than a previous formulation for the same problem. We also discussed how the formulation can be adapted to model a tree-based classification structure and to handle nonlinear separating regions via the ``kernel trick'' from the SVM literature. In our computational study we demonstrated how a natural dynamic cluster-based heuristic can be used in conjunction with the MIP formulation to find high-quality solutions for problems with large data sets. As a MIP-based approach, the approach remains computationally more demanding than existing heuristic extensions of SVM to the multiclass setting. However, depending on the structure of the data, the approach has potential to yield better classification performance, and thus for such data sets may be worth the extra computational time.

The predictive behavior of our model depends mainly on the number of separating hyperplanes~$m$ and the margin-controlling parameter~$\kappa$. Although we explored a wide grid of these values in our experiments, a more systematic study of their influence on accuracy and computational effort would provide practical guidelines for parameter tuning. Preliminary results suggest that moderate values of~$\kappa$ and small~$m$ already yield competitive performance, indicating that the model is relatively robust to moderate parameter variations.

Beyond predictive accuracy, an appealing feature of our approach is its interpretability: each separating hyperplane can be directly interpreted as a linear decision boundary in the original feature space, and the induced polyhedral cells provide transparent decision regions. Future work will explore visualization tools and post-hoc analysis techniques to further enhance the interpretability and generalization insights derived from the obtained arrangements.

Future research will also  focus on developing decomposition-based strategies to solve exactly the proposed optimization models more efficiently. In particular, Benders-like  decompositions could exploit the problem’s structure across hyperplanes reducing computational demand while preserving optimality guarantees. In addition, an important research direction will be the incorporation of robustness features into our framework, allowing the model to handle uncertainty or noise in both the feature values and the class labels. By integrating tools from robust optimization, the resulting formulations would enhance the reliability and stability of the learned classifiers, particularly in the presence of measurement errors, outliers, or mislabeled observations.

\section*{Acknowledgements}

The first author (VB) acknowledges  financial support by  grants FEDER+Junta de Andalucía projects C‐EXP‐139‐UGR23 and  PID2020-114594GB-C21, PID2024-156594NB-C21,  
RED2022-134149-T  (Thematic Network on Location Science and Related Problems),  and the IMAG-Maria de Maeztu grant CEX2020-001105-M /AEI funded by MICIU/AEI/ 10.13039/ 501100011033. The last author (JL) acknowledges support by the Office of Naval Research (ONR) grant N00014-21-1-2574.

\end{document}

%% file: plot.tex
\begin{tikzpicture}[scale=1.4]

\definecolor{darkgray176}{RGB}{176,176,176}

\begin{axis}[
tick align=outside,
tick pos=left,
x grid style={darkgray176},
xlabel={\tiny Feature 1},
xmin=-0.05, xmax=1.05,
tick label style={font=\tiny},
y grid style={darkgray176},
ylabel={\tiny Feature 2},
ymin=-0.05, ymax=1.05,
ytick style={color=black},
  scatter/classes={
    0={mark=*, draw=black, fill=blue!90!black, mark size=1.2},
    1={mark=square*,  draw=black, fill=red!30!orange},
    2={mark=triangle*,draw=green!35!black, fill=green!50!black, mark size=1.2}
  }
]
\addplot[
  only marks,
  scatter,
  scatter src=explicit symbolic,
  mark size=0.8
]
table [x=x, y=y, meta=colordata]{%
x  y  colordata
0.570517834079917 0.936533180541931 1.0
0.297086313939897 0.443011625407946 1.0
0.68645193476884 0.693289902782857 0.0
0.829776339014036 0.872332852809936 2.0
0.587099486291982 0.12691875451198 0.0
0.480496801661677 0.186847690652863 0.0
0.104801756067037 0.153036982377566 1.0
0.687258118879498 0.830823005071967 0.0
0.0420316232539744 0.0408072305842714 1.0
0.926927984033966 0.813864416867728 2.0
0.477695939905306 0.147948483762804 0.0
0.725779436302235 0.962877711698902 1.0
0.157897795533937 0.260994154023982 2.0
0.716087769667415 0.698350639956121 0.0
0.11624381733413 0.316653533855943 2.0
0.693133945748642 0.647031737540767 0.0
0.708958562102238 0.814641220052792 0.0
0.133200151310056 0.327883985730115 2.0
0.185897515026005 0.488641104813129 1.0
0.667550491453514 0.782563222799696 0.0
0.650335064594136 0.904467385498152 1.0
0.646257757826887 0.861228272092918 1.0
0.844351853123453 0.370103356094271 2.0
0.149482525637068 0.257838618011728 2.0
0.148886285150632 0.121808767423598 1.0
0.884181550126618 0.772274110447758 2.0
0.423543281564088 0.188381811558238 0.0
0.384819624185093 0.219446745051461 0.0
0.125266134724161 0.43805270654801 1.0
0.189693173771757 0.173056174241758 1.0
0.257775052346504 0.407040796559354 1.0
0.596967170503 0.881422349786749 1.0
0.0879120995603669 0.15229195993408 1.0
0.00141103888279992 0.0737585110479443 1.0
0.188386851114609 0.32758562795418 2.0
0.255342899161926 0.461075164316554 1.0
0.65202898563522 0.185919804668306 0.0
0.72107858641843 0.812858240864008 0.0
0.179873258696009 0.454596562171384 1.0
0.348333714085806 0.159842946971416 0.0
0.605555090468585 0.112283532507004 0.0
0.118104761143483 0.352693013742442 2.0
0.903276977872148 0.887673874119955 2.0
0.908433869359387 0.259087711732588 2.0
0.865162292527619 0.252911438049355 2.0
0.367506228297059 0.200316301958688 0.0
0.413549613277633 0.234517079183056 0.0
0.964313497978683 0.295659118635121 2.0
0.69799666016275 0.80961738410258 0.0
0.869714205433085 0.353508548534356 2.0
0.881695896542108 0.31219797260811 2.0
0.145608596968484 0.425991108736133 1.0
0.197956193849385 0.39705852486371 1.0
0.147333671787012 0.236893090884016 2.0
0.0825888870981186 0.114660982952822 1.0
0.160010718391729 0.420353113699008 1.0
0.237461482842452 0.44467990570756 1.0
0.138530947591027 0.149846732750584 1.0
0.0185266117859109 0.0839400884864461 1.0
0.935663445586316 0.745006279262797 2.0
0.875811988550539 0.356873692828004 2.0
0.904294489049058 0.370944620484148 2.0
0.34561355994396 0.202451511313265 0.0
0.892271499071576 0.258947200801116 2.0
0.126276140896662 0.294091198833567 2.0
0.88038422560535 0.829148531220739 2.0
0.405832400959552 0.105532423430346 0.0
0.552523393382141 0.113354273853679 0.0
0.181284467590972 0.347503564968332 1.0
0.114325115649359 0.191855569887084 1.0
0.387469465472068 0.239915842197873 0.0
0.0992887879452509 0.296297994371887 2.0
0.0676293075595478 0.150429459508966 1.0
0.0592327625724089 0.162156183530502 1.0
0.638543616765932 0.923723455330742 1.0
0.931380670506111 0.319322926680626 2.0
0.962806511599136 0.867829353522616 2.0
0.446535883463771 0.19077324564854 0.0
0.152640342033326 0.340805073836565 2.0
0.407689806693814 0.132506382249367 0.0
0.192662329100979 0.349412272285157 1.0
0.418952453327085 0.242457168520663 0.0
0.89776900302666 0.376546847550199 2.0
0.424312612290226 0.215838815969993 0.0
0.668073175152176 0.771095626891465 0.0
0.205411371882794 0.424328809920379 1.0
0.887109449756696 0.354049646565173 2.0
0.10437545285677 0.257358976941166 2.0
0.0423482437274289 0.361392370741307 2.0
0.214669475248784 0.438640420449114 1.0
0.118169524887344 0.083965766266701 1.0
0.625826949432183 0.759395836773234 0.0
0.0312225796431296 0.120383705129812 1.0
0.670986761759301 0.104103299099581 0.0
0.578192543994112 0.0903730078808346 0.0
0.172080402959523 0.111619215100737 1.0
0.685013153759987 0.902987025868023 1.0
0.440427263812284 0.124534200836004 0.0
0.132699545148553 0.367089898661998 2.0
0.170433810038882 0.510179686857276 1.0
0.202862251954539 0.263601643444094 2.0
0.184646962789863 0.262465018994192 2.0
0.15479438256617 0.417167598891634 2.0
0.660570979604158 0.0777322276016716 0.0
0.935001776282116 0.855890242285783 2.0
0.647456916629821 0.759822429208004 0.0
0.296623774635158 0.181280373830923 0.0
0.767761181199418 0.735769493720912 0.0
0.646577568705092 0.947338008801724 1.0
0.808353229793447 0.345335097530486 2.0
0.587575410591408 0.902349039934662 1.0
0.13296374465621 0.105788346592403 1.0
0.349709552074104 0.135380202791452 0.0
0.119260868881422 0.12052524068028 1.0
0.865540977968984 0.813654866200547 2.0
0.893128768880986 0.359204127026838 2.0
0.148245801371793 0.307075291039191 2.0
0.869898026855192 0.293097735605577 2.0
0.929077353096913 0.36670900729112 2.0
0.891791247937875 0.243572592859076 2.0
0.190556762533021 0.459455445583063 1.0
0.721358955694536 0.90393595139675 1.0
0.870183309538467 0.847813272624711 2.0
0.648366781463834 0.0793310198247014 0.0
0.596972526534039 0.15069593939016 0.0
0.122800475791994 0.117294815562156 1.0
0.649436460472532 0.939516005433425 1.0
0.659809507062521 0.020531242694181 0.0
0.136539992232531 0.327020043633449 2.0
0.498182654200919 0.126929335454229 0.0
0.902402080785922 0.82982117871461 2.0
0.968903446064895 0.302199930533518 2.0
0.952716962925456 0.276607347075944 2.0
0.141747768630874 0.126903163771055 1.0
0.098679768347514 0.317529465342348 2.0
0.590041648259703 0.131644449808826 0.0
0.853886134404306 0.307094626659038 2.0
0.400821844470922 0.278675810431348 0.0
0.0523578357895298 0.0947636202028718 1.0
0.179020410332764 0.497678763320908 1.0
0.640240430843713 0.188409989310843 0.0
0.150101253870113 0.267784801641021 2.0
0.707982566424378 0.917635041072501 1.0
0.649166270307502 0.748851477462861 0.0
0.849113405359854 0.376368190959955 2.0
0.945652259309795 0.855682621551196 2.0
0.139704784327102 0.0185851238929413 1.0
0.67603063271203 0.716070180270593 0.0
0.963382831017714 0.375922624448971 2.0
0.0175162490473133 0.0289699332563937 1.0
0.569163902475588 0.951275905906117 1.0
0.609805851588797 0.996387948669098 1.0
0.405567084612016 0.194571148298675 0.0
0.526503951560906 0.0731207161879294 0.0
0.869612122746349 0.262970999823812 2.0
0.700939507015472 0.810464060425487 0.0
0.208551344077066 0.442931259082878 1.0
0.160139214796047 0.163042878623655 1.0
0.0730926243271297 0.145998908440239 1.0
0.917429624995143 0.821364125692989 2.0
0.307667160486204 0.268241960664983 0.0
0.850940303048711 0.93450478139203 2.0
0.120548163852158 0.124362587058947 1.0
0.651918607176656 0.808538414988572 0.0
0.311104517477295 0.177243405724683 0.0
0.655937980451128 0.736083104849885 0.0
0.879004661846147 0.808041543120086 2.0
0.343747083069573 0.232878525611525 0.0
0.5824579335763 0.113721080312821 0.0
0.296918251134902 0.440854554962473 1.0
0.576879646512845 0.129083918904599 0.0
0.782653191554247 0.793655518024313 0.0
0.0397450839070944 0.401278401240023 2.0
0.997917893277752 0.906527298199168 2.0
0.911527679120225 0.325424108987263 2.0
0.595342477254015 0.124266091879584 0.0
0.320428694878976 0.433657257591024 1.0
0.278235005038234 0.316528487639059 1.0
0.206392125979507 0.28540944321559 2.0
0.935750035931752 0.243463998879546 2.0
0.42856555850009 0.201513865785513 0.0
0.237537413732167 0.555611859290932 1.0
0.659148986330943 0.694410523296737 0.0
0.914017508207169 0.774168674063727 2.0
0.0539735546023098 0.124389121351942 1.0
0.611386485833897 0.096124073413384 0.0
0.606893786021481 0.82739572175994 0.0
0.56672731273054 0.951840426614412 1.0
0.368403148493608 0.237355347495205 0.0
0.257891916514652 0.468909857703979 1.0
0.0914325038667052 0.154300999342346 1.0
0.650334447897054 0.103901289729453 0.0
0.670901570311169 0.792124011673895 0.0
0.673851058413749 0.0772603919881176 0.0
0.16838211019379 0.26168079272873 2.0
0.690256770536079 0.950438675791975 1.0
0.586347759089581 0.931040463950336 1.0
0.349733212959186 0.205089104876509 0.0
0.144512446241536 0.352292002634177 2.0
0.0481586980523458 0.0547651471527983 1.0
0.0783977448642537 0.370737113040118 2.0
0.623265180652454 0.1977373728314 0.0
0.583891760903868 0 0.0
0.615508755750091 0.12454212843652 0.0
0.929218933959961 0.805301611898814 2.0
0.214263635657801 0.516962352148385 1.0
0.948464192480227 0.832968628335693 2.0
0.0607885091967934 0.0904677950194147 1.0
0.921790082238231 0.80476530678313 2.0
0.591335209271014 0.12909700542567 0.0
0.966410994355393 0.916795372238703 2.0
0.411105693970644 0.241561477813138 0.0
0.155249996988544 0.319313450788287 2.0
0.91071129549821 0.787002584384533 2.0
0.0656522461418085 0.257326374696503 2.0
0.139569142321504 0.329623553131982 2.0
0.0284856111353635 0.223014567937796 2.0
0.199883390477488 0.465525715855395 1.0
0.643643772031327 0.144191161450995 0.0
0.127606950217836 0.105633293431805 1.0
0.861099832414754 0.810519143694518 2.0
0.418401819348114 0.184364599805325 0.0
0.0457625651475828 0.105558094289049 1.0
0.667083264427147 0.866921671303273 1.0
0.882937072393862 0.719063006187039 2.0
0.678284337674791 0.920547900696993 1.0
0.12744487262794 0.301234467743388 2.0
0.722459880217595 0.831158872296196 0.0
0.162090901065319 0.308364918350852 2.0
0.679456374247228 0.846910700999449 1.0
0.194787096764013 0.325202219593951 2.0
0.11417509217215 0.335924496184533 2.0
0.669415230410552 0.725052554940436 0.0
0.675087954739465 0.90018892450599 1.0
0.385636442096383 0.247508168691885 0.0
0.112327190542852 0.290521401980798 2.0
0.240484437124596 0.471190184483674 1.0
0.663020846848041 0.868276689618249 1.0
0.185071454666866 0.437403247996141 1.0
0.17155716775342 0.289847114131581 2.0
0.912411775013446 0.830611873575771 2.0
0.823216640046678 0.320929041828332 2.0
0.385947819369948 0.127014926011538 0.0
0.407296492917094 0.114520866349667 0.0
0.667492741141441 0.890219732468951 1.0
0.161981638263834 0.431864301399395 1.0
0.739962166012029 0.755265634967907 0.0
0.180285601918005 0.364234589091665 1.0
0.83857743688764 0.360782307159919 2.0
0.719745610566859 0.92044402606041 1.0
0.0967521074363307 0.150516879816661 1.0
0.807950027321959 0.321097686865056 2.0
0.212571435601142 0.371133573513265 1.0
0.174959138061159 0.311955734368373 2.0
0.123911495204681 0.134833387391519 1.0
0.847223645416094 0.805290978035258 2.0
0.113476256455252 0.126451351276023 1.0
0.920232033663464 0.774861517329765 2.0
0.839286272184179 0.2900334963864 2.0
0.887518388625677 0.756988348719383 2.0
0.92497536230642 0.838570254741184 2.0
0.0465131077160012 0.284083348364764 2.0
0.705120132448865 0.738845115340115 0.0
0.471688896400994 0.126444427175921 0.0
0.833645826794055 0.314265134061414 2.0
0.65929479949009 0.0692568615946791 0.0
0.717765051450275 0.879940853804189 1.0
0.265708587272265 0.407255165649759 1.0
0.232629972838677 0.459419495605275 1.0
0.605914936586394 0.055605724414402 0.0
0.712438312509707 0.820521015549327 0.0
0.631588376553562 0.905588504555923 1.0
0.359248803312276 0.143588150299291 0.0
0.887539328512771 0.811916293761296 2.0
0.641074572821504 0.19155802054943 0.0
0.598092905444807 0.114592400833911 0.0
0.738245045290802 0.716188923411856 0.0
0.789792599389368 0.697606324057743 0.0
0.417618130449585 0.124304963449253 0.0
0.658718542851912 0.813705719642869 0.0
0.886716656808354 0.323451970899955 2.0
0.848183493661128 0.822189731028241 2.0
0.770506873723818 0.0736882511267232 0.0
0.487599635961735 0.236409870380052 0.0
0.135899994292054 0.243628154669604 2.0
0.744168305548423 0.950132831131365 1.0
0.255183585127421 0.325510827447328 2.0
0.893356760149483 0.305324211010584 2.0
0.40056952893171 0.141357143705157 0.0
0.529407499555546 0.142786553324291 0.0
0.191738265863587 0.305134483966317 2.0
0.67883716597242 0.897354739190053 1.0
0.43161772424014 0.224893316753929 0.0
0.130518339633627 0.126625350458613 1.0
0.183528850878698 0.444075490043341 1.0
0.176997622732715 0.242843814737729 2.0
0.771059067434314 0.789813642971436 0.0
0.798437477616852 0.338043819351964 2.0
0.700524238794647 0.722310009889137 0.0
0.872379939175581 0.359417558057384 2.0
0.949683616861477 0.373646240120574 2.0
0.120638721404672 0.520780467935998 1.0
0.394456857589471 0.15456025941131 0.0
0.873012818688384 0.266630999699917 2.0
0.0886487959287889 0.323990589099819 2.0
0.919280812575769 0.796018420776055 2.0
0.663869312582003 0.100653267824604 0.0
0.683018496747801 0.896098478967222 1.0
0.670868967589539 0.806312430892496 0.0
0.681530278597351 0.675723258570613 0.0
0.0779632806548833 0.112502434810152 1.0
0.55846200379803 0.156766877463768 0.0
0.642233453172298 0.0853079738924166 0.0
0.20555975822369 0.455845027804366 1.0
0.116491631620095 0.275826641219441 2.0
0.173246298472933 0.212995882593039 2.0
0.568592143985401 0.0587076947758313 0.0
0.860431401836355 0.333405453553222 2.0
0.0728495327440353 0.303278317684114 2.0
0.611573287156118 0.120224534451666 0.0
0.89766466164595 0.874108667887838 2.0
0.707820242982955 0.976511288279692 1.0
0.172010055259671 0.264722365129487 2.0
0.825409061999086 0.256834757564078 2.0
0.868987086805781 0.309278999996875 2.0
0.136938089552812 0.477082524847054 1.0
0.258089739354068 0.486364907288409 1.0
0.838175202552872 0.301768486165535 2.0
0.137532012336665 0.342537603182414 2.0
0.920052632769895 0.352640825269747 2.0
0.530039395777766 0.887362856341398 1.0
0.105810330933034 0.518457561185541 1.0
0.703928557196249 0.666996993694755 0.0
0.680042254563971 0.729091170676285 0.0
0.608066636527561 0.106362152243895 0.0
0.222614561173118 0.468214676546921 1.0
0.310849172512448 0.440483744123313 1.0
0.650439227116627 0.724252470059223 0.0
0.642045533560027 0.85280249136704 1.0
0.712097956990523 0.75733205918487 0.0
0.653297584907938 0.979364625966231 1.0
0.472089939444538 0.196108933701071 0.0
0.660996370417476 0.957934344680822 1.0
0.645277640415268 0.629000708445268 0.0
0.430430398999474 0.178672912229313 0.0
0.224769531387124 0.390926457296562 1.0
0.707928981768592 0.724269704558029 0.0
0.897632236604033 0.344124219844603 2.0
0.969445725438124 0.328621425769863 2.0
0.629875308531162 0.70515089861971 0.0
0.657530150937089 0.788603504388059 0.0
0.199480987456923 0.458761269075518 1.0
0.734254202059232 0.689129785742857 0.0
0.221356289970845 0.353410398253706 2.0
0.101084930261103 0.326741962615186 2.0
0.610769021599537 0.0782112978308417 0.0
0.191723844258446 0.36013656187374 2.0
0.876738761572747 0.820244098891521 2.0
0.0967265412543396 0.167011376517391 1.0
0.073248946843212 0.176861197601573 1.0
0.432878298244759 0.125804909944709 0.0
0.195983196298103 0.356471796403483 2.0
0.647902761561074 0.124573426389015 0.0
0.632761843783954 0.0333638140061434 0.0
0.948916974695672 0.759488038758099 2.0
0.195277991157335 0.412364515647489 1.0
0.254387514185906 0.472484049230542 1.0
0.959784055687318 0.710876385949989 2.0
0.629366199762551 0.0910823092168333 0.0
0.578297571109678 0.941853314049998 1.0
0.897776925058138 0.320232741191024 2.0
0.203842222877774 0.479171545518934 1.0
0.901398627319904 0.361819527036696 2.0
0.644590883233031 0.0947768065318178 0.0
0.105617358484076 0.0687968182244243 1.0
0.128857806902498 0.0357095291413517 1.0
0.950669677709634 0.282858714947876 2.0
0.901666109204785 0.886137316112575 2.0
0.186730521483853 0.491627306032422 1.0
0.654067138387622 0.901238064701126 1.0
0.887917937515766 0.823160223693665 2.0
0.673522697318132 0.197713265873416 0.0
0.943018686017955 0.810077995452558 2.0
0.870727688051015 0.862790452345517 2.0
0.630384727079106 0.978841806816566 1.0
0.866022512592761 0.768020959740981 2.0
0.628720024847091 0.945987401470689 1.0
0.100406235738061 0.173478759778158 1.0
0.834760182546248 0.849812453541795 2.0
0.623843627763289 0.130645773658013 0.0
0.698293583903565 0.765078715948008 0.0
0.62484204444966 0.0265686701105343 0.0
0.637425962823425 0.136737030010329 0.0
0.870900942107064 0.368427112965158 2.0
0.27456815936941 0.488572900244292 1.0
0.202165727002212 0.518159779344026 1.0
0.614411431004103 0.948881629955996 1.0
0.568281118196133 0.86668032021 1.0
0.918670611729279 0.275068226039142 2.0
0.67532228136736 0.940618707938495 1.0
0.596736592929773 0.12087578550259 0.0
0.915269333359282 0.884867608901017 2.0
0.913609654976598 0.826851707689181 2.0
0.97789356671271 0.84385596668998 2.0
0.0691255733780684 0.145394386076019 1.0
0.143942738396445 0.265917746903895 2.0
0.0565072693187007 0.0708545489544286 1.0
0.738617789188301 0.748605158638522 0.0
0.42177888389686 0.162259538744443 0.0
0.201048630122474 0.444845491978806 1.0
0.395777870173051 0.192983458832926 0.0
0.637520919565958 0.14818711873548 0.0
0.0971413422658456 0.136671387447505 1.0
0.07785163038792 0.114478199115118 1.0
0.886581192365302 0.366718874871021 2.0
0.0379751685699183 0.0968238616168636 1.0
0.611535010076386 0.15912025621346 0.0
0.245323143413426 0.45232792548089 1.0
0.891084217403226 0.244066459820987 2.0
0.725056944849636 0.696285551944855 0.0
0.697904587053373 0.921287504143776 1.0
0.869502239604658 0.806683896484219 2.0
0.0611532383648898 0.350266864226188 2.0
0.54978452506807 0.0553520723531205 0.0
0.0853044910626532 0.127066297348757 1.0
0.690435481001865 0.760768744827693 0.0
0.175015413672093 0.452126862082905 1.0
0.194064537432944 0.288596460127967 2.0
0.0918766775129992 0.152582880635132 1.0
0.221602565663509 0.398411520496227 1.0
0.726810038841076 0.870915014291986 1.0
0.398011722721676 0.194564306177681 0.0
0.608218513442838 0.0496391415875108 0.0
0.0977386057675045 0.29787252714076 2.0
0.911498238729757 0.362609080059578 2.0
0.372401007822459 0.157054407838826 0.0
0.270356831324293 0.405228050880615 1.0
0.0728052804470008 0.0898696139045632 1.0
0.865682599736476 0.363458530923749 2.0
0.351110334896829 0.146224653546763 0.0
0.200234177771778 0.380256627272167 1.0
0.619273386710731 0.0884233130055247 0.0
0.117786032355472 0.362948545724072 2.0
0.417133379745322 0.159097953505647 0.0
0.12883260888668 0.325600221758244 2.0
0.639587046082329 0.859979403910216 1.0
0.661377250684431 0.687536883033722 0.0
0.979326502061075 0.842381586682645 2.0
0.360825282686166 0.235260953323418 0.0
0.706478153771879 0.734058431699333 0.0
0.839468185731093 0.805577063472843 2.0
0.0357581435893273 0.0256379673881809 1.0
0.125885435518723 0.0857099082605887 1.0
0.235226402470594 0.385953407071086 1.0
0.680854439984991 0.765579563250359 0.0
0.1960936314597 0.480126983428769 1.0
0.0144424730181688 0.0517594085418119 1.0
0.0823794147455468 0.413304131149703 2.0
0.0766810888054891 0.33127208277633 2.0
0.157610535766671 0.0913442892689423 1.0
0.612091277935979 0.120614102096793 0.0
0.674556384153073 0.849018967301308 1.0
0.337946451386081 0.0855855593983335 0.0
0.27625079543059 0.498413627362275 1.0
0.29794544230441 0.424316360701238 1.0
0.43732286916357 0.196551459870374 0.0
0.189967739961787 0.444372708871666 1.0
0.779972625968845 0.787029851249533 0.0
0.596046357688342 0.158546237486751 0.0
0.884705095093177 0.360430712317747 2.0
0.847779921250728 0.331102840525966 2.0
0.211014633231204 0.365210323497178 1.0
0.956327460303408 0.868395066089485 2.0
0.703131459997965 0.868587852881358 1.0
0.129877797668259 0.202216922973806 2.0
0.631803418140094 0.101184888046249 0.0
0.92743017780679 0.787372739175739 2.0
0.757313991467312 0.764369237082798 0.0
0.16461133980889 0.325342443400226 2.0
0.0903217915647123 0.168217249935689 1.0
0.605721521112026 0.0800156976879349 0.0
0.414568929360355 0.140218802846219 0.0
0.937873139729551 0.407911064555314 2.0
0.71823157496714 0.77686856022191 0.0
0.121770750293456 0.323707005445511 2.0
0.14231479795327 0.131181852566007 1.0
0.918695639156464 0.752841383124265 2.0
0.159033624679593 0.526837217870061 1.0
0.683028358221624 0.943896802824218 1.0
0.720889678497818 0.893745777867296 1.0
0.650848502346853 0.727046914316144 0.0
0.445632734905269 0.171156729299575 0.0
0.570406686203479 0.161032827886855 0.0
0.224146045054906 0.463735535490177 1.0
0.637027120202073 0.771666048800262 0.0
0.3888301881452 0.213924991215433 0.0
0.332808108837995 0.213789570693339 0.0
0.651279106023805 0.121534904332869 0.0
0.0750187404250752 0.119090718342567 1.0
0.705768009874022 0.660447870908175 0.0
0.741101704639358 0.80182166948374 0.0
0.670643269523777 0.930294891708673 1.0
0.586401175089439 0.127326720571374 0.0
0.42528648636971 0.196642779408315 0.0
0.581234773162718 0.133642994208321 0.0
0.900301898469476 0.393483821260759 2.0
0.326388558173301 0.230502062994569 0.0
0.0819277535175413 0.102476446792882 1.0
0.882990637131386 0.248245991698433 2.0
0.928044472009534 0.885688706191071 2.0
0.63779069892145 0.763485001309803 0.0
0.936561005436409 0.849836531033495 2.0
0.654301747585741 0.743567551832576 0.0
0.237215452614857 0.473775568412453 1.0
0.723054488946037 0.734965974301382 0.0
0.826913148693731 0.28386907574772 2.0
0.783443390271656 0.291632665112012 2.0
0.591474928115369 0.0303423448858554 0.0
0.156937322680108 0.329375299869254 2.0
0.245717418247087 0.41331814196974 1.0
0.757464507980597 0.820216351774952 0.0
0.095684591026226 0.0593818201744288 1.0
0.661657395243936 0.915078780551585 1.0
0.14218378657326 0.124409777508026 1.0
0.5826968431126 0.967872256107244 1.0
0.0423520552218357 0.190007999999609 1.0
0.648136386890498 0.0449148699936877 0.0
0.887628102433961 0.868512927517506 2.0
0.557594138646875 0.0384138674218461 0.0
0.636746328967129 0.932667584064077 1.0
0.614853826481765 0.896476934044008 1.0
0.901533221675864 0.863769236738895 2.0
0.878990757760517 0.346575858822516 2.0
0.70365489404422 0.712884718916835 0.0
0.849924607002497 0.321281073136502 2.0
0.130522430752007 0.260565565925059 2.0
0.907212826713259 0.837877007111473 2.0
0.06083739813165 0.133615247984355 1.0
0.67501816706859 0.995748211722355 1.0
0.860107066555835 0.860313408586906 2.0
0.891171944580053 0.878066136132743 2.0
0.114654249673621 0.127882042123749 1.0
0.937448934966441 0.431020140844854 2.0
0.931378807389255 0.770973589961811 2.0
0.970534471164131 0.798054628615922 2.0
0.832996179817796 0.835436905500806 2.0
0.335892658511607 0.226748754939007 0.0
0.876473186985249 0.764061965639052 2.0
0.58440008706798 0.0826742482538904 0.0
0.684400277547439 0.0977160535353542 0.0
0.0794856592398108 0.140259349018463 1.0
0.671874218272615 0.103513435863622 0.0
0.34156829964254 0.186634680986187 0.0
0.203688751059562 0.281845853455822 2.0
0.692965954787856 0.901641901015893 1.0
0.943334409569192 0.279330770093217 2.0
0.736239163831783 0.750543496731402 0.0
0.0379055340058214 0.068926652410192 1.0
0.895310701812695 0.326333644754605 2.0
0.766523392829459 0.833833875486923 0.0
0.269747474877367 0.356588952135527 1.0
0.639780990291127 0.10917306113836 0.0
0.657588427372249 0.772629262133589 0.0
0.855196262135285 0.808495894115898 2.0
0.622182891005694 0.964508035976551 1.0
0.918160254548362 0.822009616989601 2.0
0.764157172737355 0.742848588828439 0.0
0.680134564113159 0.0812023825809932 0.0
0.876088956372469 0.817754307707616 2.0
0.143637433345049 0.194266682757026 1.0
0.668803007714916 0.0845226286353628 0.0
0.525132675531374 0.113674618793078 0.0
0.453863578375059 0.143399462246681 0.0
0.935282579217335 0.460596028167985 2.0
0.11914146737285 0.135256012820006 1.0
0.933815983304611 0.721896503766691 2.0
0.117006166453925 0.111650529483206 1.0
0.939714506853669 0.795785803873706 2.0
0.375376548845672 0.219549382740231 0.0
0.885921963915533 0.341992019754135 2.0
0.912709846857871 0.83993823409234 2.0
0.219287603072638 0.432280675514896 1.0
0.16699056112032 0.428161939330575 1.0
0.722911870268782 0.761364169268474 0.0
0.112346611602197 0.111700919908891 1.0
0.528345135615165 0.10187874928622 0.0
0.928950340626544 0.292070467688621 2.0
0.944811054246375 0.803467937493295 2.0
0.887358962395453 0.879546432140296 2.0
0.621769903126668 0.849510057177434 1.0
0.93578273987822 0.239265649468226 2.0
0.880067338493847 0.298062340943836 2.0
0.668996431318965 0.91404086830699 1.0
0.483131385892503 0.134246448856362 0.0
0.684496260648505 0.115578529977514 0.0
0.0795015923403875 0.299905212649709 2.0
0.101492744070977 0.350208423787143 2.0
0.678047822510655 0.965879868528955 1.0
0.100214385863591 0.0969650091975667 1.0
0.128779729475198 0.162614198988009 1.0
0.803780667495815 0.722030082314991 0.0
0.617779782468675 0.718418515506826 0.0
0.716374602791329 0.895903945525196 1.0
0.849882996164764 0.283727204881545 2.0
0.314445545259854 0.45020585255332 1.0
0.663813767719678 0.746929284730461 0.0
0.280797520024739 0.429625644835547 1.0
0.802969562073202 0.675777839943299 0.0
0.337968986661686 0.175744737705661 0.0
0.213399509875492 0.438565119708329 1.0
0.882150810349758 0.329213048609234 2.0
0.0811807673034106 0.0548739039242181 1.0
0.62600769906289 0.0556946195239221 0.0
0.379904521623581 0.19096768886641 0.0
0.656923791997609 0.935771507360311 1.0
0.622281660421138 0.0537484403522491 0.0
0.268094176690796 0.481844206660197 1.0
0.679242437589851 0.941111070399868 1.0
0.691403700082096 0.925471330970546 1.0
0.16449577451626 0.0693015891551935 1.0
0.759833884890198 0.400022499459006 2.0
0.696486021420491 0.930283816819631 1.0
0.701971932985268 0.908119762635713 1.0
0.0475832443761849 0.126600727552367 1.0
0.385110239156813 0.188497713454878 0.0
0.421476240174504 0.184984704683556 0.0
0.0817940814476456 0.184559734142671 1.0
0.565971623232665 0.944117076206686 1.0
0.137142594677561 0.199299427307468 1.0
0.637644845078967 0.874105479958624 1.0
0.928887270016448 0.292886724300381 2.0
0.718076877010963 0.765524563740051 0.0
0.728688150404057 0.701227251781714 0.0
0.650485176323171 0.926678975814571 1.0
0.138215192434325 0.409144687445431 2.0
0.117636971698935 0.167003695414411 1.0
0.425112386473398 0.168785921845256 0.0
0.350369736157782 0.193656060992529 0.0
0.354714723128163 0.0801543623280336 0.0
0.447291492380411 0.0880224873771582 0.0
0.797740754529629 0.364336242652966 2.0
0.12334239205081 0.0971610153883555 1.0
0.159115712943728 0.344220212605539 2.0
0.0654752445649177 0.12356651216522 1.0
0.712891165434545 0.774280235786084 0.0
0.0794192316748181 0.293095344616944 2.0
0.885219684163111 0.222150878610222 2.0
0.581328012338021 0.0659967403396159 0.0
0.58090257253556 0.155818497401307 0.0
0.673007909339269 1 1.0
0.401876964706696 0.202706752953099 0.0
0.12680265393856 0.318454681841306 2.0
0.623230129382962 0.898275432593633 1.0
0.661368241696642 0.870151635997876 1.0
0.599641636467116 0.930456957129079 1.0
0.233934505799145 0.308549814108158 2.0
0.105935556409445 0.367009565175804 2.0
0.130374804206458 0.347406470114448 2.0
0.683253004739446 0.0755569128557111 0.0
0.690479670476559 0.757419505534344 0.0
0.210938958163101 0.355820379516933 1.0
0.737377035009568 0.717195246906905 0.0
0.142947341311326 0.325746210143888 2.0
0.625756193072108 0.117860578071615 0.0
0.915351114136617 0.887805954622834 2.0
0.312086606522613 0.1393335220556 0.0
0.606231732466045 0.119296466676877 0.0
0.0992603484047573 0.298965169780643 2.0
0.749635457957033 0.817273803129296 0.0
0.913078500158596 0.360946007448416 2.0
0.256571861970563 0.43878001124475 1.0
0.112628350068632 0.127086647694582 1.0
0.122129478406254 0.293585121308118 2.0
0.171365396434561 0.444426376394817 1.0
0.893030159274498 0.783967521116872 2.0
0.269941936711295 0.418887432968091 1.0
0.181602438189087 0.0551218888370639 1.0
0.0943268122569845 0.25155599001123 2.0
0.899711524386174 0.229521426287085 2.0
0.171811332509131 0.226647177038748 2.0
0.451316385477377 0.161455567567866 0.0
0.88940197429791 0.326625323983154 2.0
0.204575911889252 0.463760036141997 1.0
0.0691343191096175 0.157090228036696 1.0
0.170504444021013 0.301339518436557 2.0
0.864906341120328 0.855245670108028 2.0
0.959975902598077 0.815934690888015 2.0
0.113603507988477 0.114020423152591 1.0
0.844407803791289 0.76422490983313 2.0
0.363357611822044 0.194918958925341 0.0
0.197131857641605 0.487205423875809 1.0
0.617282959116864 0.0373493490803247 0.0
0.650882727123684 0.692135719477485 0.0
0.143597937124189 0.278985256353393 2.0
0.91326083395217 0.798919724114412 2.0
0.969893630328572 0.816007525572203 2.0
0.236169559495259 0.410078992912726 1.0
0.139790911076039 0.231008429200657 2.0
0.130485167155013 0.166052746755994 1.0
0.876436439412618 0.416401668144513 2.0
0.600396881298878 0.0710021809412103 0.0
0.0977213873951637 0.0816750271287513 1.0
0.84943896353243 0.794592077053909 2.0
0.263983797643123 0.436849252611106 1.0
0.23763453631134 0.45623112316699 1.0
0.0725982055163032 0.121356546762305 1.0
0.829553402063838 0.351964398567838 2.0
0.698338684301098 0.142947394269541 0.0
0.288857636464292 0.486736298315535 1.0
0.949117082317575 0.817165146969319 2.0
0.684326197512781 0.814992877582841 0.0
0.823980111764429 0.920761807310263 2.0
0.617399616002806 0.0447794683821696 0.0
0.0755040049149662 0.0977858170824921 1.0
0.876420074377949 0.346183157148566 2.0
0.837129507259971 0.895489931673713 2.0
0.924722317986706 0.387915122004049 2.0
0.853785034922348 0.802569004213102 2.0
0.184451327818846 0.364226091427402 2.0
0.214926579225492 0.298510686294129 2.0
0.194022084407371 0.43392942694005 1.0
0.979207570585748 0.78359330590143 2.0
0.534127853712752 0.096329828417881 0.0
0.601081061113989 0.189118110943257 0.0
0.92109089513338 0.324021168939977 2.0
0.958535751467536 0.798559868453689 2.0
0.343268450490052 0.457094614101042 1.0
0.0761128539057502 0.0739661633348758 1.0
0.87515749571448 0.829870549582657 2.0
0.331305637688229 0.252080970558914 0.0
0.161128101207088 0.31332335182242 2.0
0.619529751790346 0.746096635425876 0.0
0.176613153550521 0.354993704338068 2.0
0.73197048616657 0.78366033962461 0.0
0.660031819672939 0.778316631892557 0.0
0.665814296949485 0.958252980876806 1.0
0.682714350853013 0.756508286018264 0.0
0.631172408565967 0.0874558944671402 0.0
0.0380739619273099 0.152361613222878 1.0
0.427652986718174 0.141987592621786 0.0
0.136295985193522 0.0647663889866347 1.0
0.634929983452725 0.934266321309809 1.0
0.885921815264614 0.82984754037046 2.0
0.641170970170876 0.912446831458255 1.0
0.897250537141793 0.820722302576777 2.0
0.356160221241444 0.167749826036453 0.0
0.223095959042651 0.447742571396633 1.0
0.437094712970426 0.137400134673086 0.0
0.226869724511175 0.495556598871721 1.0
0.891570521830575 0.853616089798425 2.0
0.119056467703541 0.339498854828762 2.0
0.603284783530556 0.144197396992877 0.0
0.625602875603731 0.928688809718744 1.0
0.922237529446378 0.343659562117636 2.0
0.67057726518322 0.871961050957686 1.0
0.690542762844773 0.761374997015162 0.0
0.588136645018114 0.0502518244310591 0.0
0.0261259184247577 0.110159479093783 1.0
0.596916985342948 0.107435956371127 0.0
0.573699503495271 0.135324879028389 0.0
0.342034686891571 0.189416477195947 0.0
0.691349577061501 0.707173934037339 0.0
0.91016166010862 0.879713795701155 2.0
0.600982688090496 0.924661176585776 1.0
0.736546298844654 0.698993936717516 0.0
0.617554065774485 0.0463374452671454 0.0
0.621348942016661 0.140466891762672 0.0
0.877689161258623 0.805540651176551 2.0
0.719804877862198 0.750066730311387 0.0
0 0.127864671194034 1.0
0.186997511477003 0.455116951546009 1.0
0.111211148295644 0.324534948615853 2.0
0.906544357880567 0.315552496441222 2.0
0.7325786653047 0.879677547127452 1.0
0.0656440505947761 0.141030636412431 1.0
0.928504889640526 0.823852506227116 2.0
0.621995709145352 0.914025779854484 1.0
0.728535295115573 0.743918423997775 0.0
0.903532802217399 0.82308581962441 2.0
0.0858328866636007 0.251899125223795 2.0
0.107628454316718 0.1049667164514 1.0
0.375366508639813 0.108424904846349 0.0
0.661036572128016 0.793254256970026 1.0
0.17338359954839 0.338176715843103 2.0
0.880569901563964 0.344593061040373 2.0
0.712239741042214 0.893448319381384 1.0
0.596966228003048 0.121403439874795 0.0
0.770399575480054 0.702950524972418 0.0
0.113258935773264 0.288605162353385 2.0
0.218930805796768 0.364876297424478 1.0
0.160627794059964 0.453919935241936 1.0
0.137556325253376 0.28004224057562 2.0
0.670493853375102 0.897163114996226 1.0
0.607847896119503 0.133661728543668 0.0
0.883728414011496 0.311861144395672 2.0
0.437849207731187 0.171732010343849 0.0
0.196284562570225 0.357363373327136 2.0
0.673104638584085 0.0306519416435269 0.0
0.642063763951708 0.981335015564489 1.0
0.0533866504676562 0.127167094878194 1.0
0.39433321204459 0.164386724612289 0.0
0.607289592601142 0.111495904297261 0.0
0.273919910297334 0.392386214931912 1.0
0.678209124110022 0.784105765096611 0.0
0.695552594288829 0.672112132962778 0.0
0.408137495454997 0.187503850196578 0.0
0.0931346599441967 0.124117065894954 1.0
0.142236584022752 0.0805159443207926 1.0
0.234249589600336 0.497175699770322 1.0
0.393054612070819 0.137963555284181 0.0
0.77133853801917 0.885848392154039 1.0
0.588795857342414 0.909651504392124 1.0
0.156927465846718 0.403155986147396 1.0
0.680269526705125 0.761337109551208 0.0
0.924914148722098 0.834261385340617 2.0
0.453694284967188 0.234782885809971 0.0
0.790794847829838 0.367289816569652 2.0
0.284577148760186 0.199551282023238 0.0
0.209760701776789 0.464805134502277 1.0
0.6285690239963 0.866788598724803 1.0
0.769907781670137 0.913849662294004 1.0
0.0535344320814744 0.123696834252682 1.0
0.211470580993938 0.433261579723302 1.0
0.141713081227611 0.367089623528887 2.0
0.0666582902269343 0.285560455448532 2.0
0.613571803897838 0.0596955343675371 0.0
0.343572138577937 0.175146832823973 0.0
0.607838128624927 0.0825761662531855 0.0
0.453681748239513 0.221316775525454 0.0
0.88855403353578 0.807631032999696 2.0
0.715283375037955 0.959233343608219 1.0
0.800307025333204 0.313967676262875 2.0
0.717095188426178 0.798749473225766 0.0
0.840365143163614 0.283629613881229 2.0
0.417847748995451 0.181936537508209 0.0
0.35293117136842 0.186100163855323 0.0
0.0685185945635348 0.133869621124924 1.0
0.389226308584966 0.167584628129037 0.0
0.262518582272656 0.482735355914332 1.0
0.839509194124375 0.338576227795029 2.0
0.182373383511137 0.322221515959149 2.0
0.201005902259424 0.50796234826644 1.0
0.639811887078971 0.0181000272630384 0.0
0.0793684927456094 0.308296201969915 2.0
0.611033484009307 0.0746932320663282 0.0
0.932052431699869 0.860966474350929 2.0
0.54576678784686 0.249644232687708 0.0
0.68718201793118 0.871816533211764 1.0
1 0.304957605104006 2.0
0.557985237318248 0.970599536328041 1.0
0.337983112647267 0.5043276318537 1.0
0.357610506049011 0.163088020478746 0.0
0.151761954055555 0.445946306325022 1.0
0.359397087524706 0.224439169846824 0.0
0.41853655019742 0.117786697598136 0.0
0.12853302924352 0.33840730838982 2.0
0.894701966336455 0.297709835046814 2.0
0.39381969370036 0.104972970296842 0.0
0.175447344821088 0.256146626207115 2.0
0.415629839991219 0.134118346442538 0.0
0.910273131482129 0.370322838635571 2.0
0.0604387398438494 0.15968341751661 1.0
0.403817369230502 0.197916384707898 0.0
0.915864089597657 0.796847769870973 2.0
0.614995070060394 0.94951555701959 1.0
0.805764413515601 0.85600308438474 0.0
0.87802704947584 0.282050080940749 2.0
0.674133360356067 0.0967211743918445 0.0
0.907234347828938 0.319473813071295 2.0
0.683293798094925 0.880499981640966 1.0
0.496405279014557 0.259262408891105 0.0
0.694717018229522 0.884248201563353 1.0
0.703451802335899 0.790310731044261 0.0
0.431472588336587 0.123642078684862 0.0
0.354235979934208 0.19976594022057 0.0
0.212744357303025 0.419068703239604 1.0
0.906357671194234 0.368562688161757 2.0
0.763893661935551 0.774749959236388 0.0
0.252443505086084 0.444379832656179 1.0
0.219997968860281 0.444279399656438 1.0
0.100822165119767 0.0186810144786813 1.0
0.572701623235888 0.0947677184858688 0.0
0.40198470169248 0.176324846259767 0.0
0.879486096496221 0.846535779149458 2.0
0.801031829210274 0.827484850667698 2.0
0.0513328708391539 0.353350795924924 2.0
0.281854477057304 0.527586331996358 1.0
0.66290347837189 0.730988149756978 0.0
0.886030621099355 0.293330792345838 2.0
0.831535216603094 0.343899908193048 2.0
0.692109141570537 0.965314099695931 1.0
0.371934061044142 0.175737776081133 0.0
0.197156674555769 0.521280705983992 1.0
0.188476103143511 0.459150362021949 1.0
0.922687897543927 0.779175556136378 2.0
0.620952942650557 0.891958838580795 1.0
0.696398456117463 0.955420651714471 1.0
0.932450374454084 0.862151945624243 2.0
0.646911878832363 0.894823315395557 1.0
0.127830587205298 0.273865597568344 2.0
0.869029585958985 0.838426850944829 2.0
0.337653084028883 0.463649603054385 1.0
0.0495212288479992 0.104378380129655 1.0
0.0982581034301777 0.108628750200153 1.0
0.717602187450579 0.737839808972082 0.0
0.0805332775639139 0.314415564556259 2.0
0.756688784919247 0.763940653567344 0.0
0.107918082214655 0.199401110649606 2.0
0.599981658744607 0.0764450108259979 0.0
0.13479841042569 0.327334252182558 2.0
0.562739298044026 0.0681893100846122 0.0
0.706773718705061 0.790037238089991 0.0
0.400148881368074 0.212337766102262 0.0
0.828751333929504 0.831245660111936 2.0
0.365657831230605 0.115508279154919 0.0
0.697457796136646 0.832935041104771 1.0
0.0579362201587648 0.171810727762546 1.0
0.410322173608136 0.191823545899418 0.0
0.615683228398219 0.896394851940847 1.0
0.610317885176236 0.153367054085865 0.0
0.653734352311782 0.929292247001182 1.0
0.959861661061158 0.400631227203782 2.0
0.640758806990859 0.792971079488796 0.0
0.628432780722482 0.930524699234298 1.0
0.953859179083953 0.849280880506715 2.0
0.134864770164574 0.358697498722732 2.0
0.186835455615414 0.458706746730298 1.0
0.613759031026063 0.115203212746093 0.0
0.372303659547162 0.144023820386544 0.0
0.0910424816825735 0.268298015663909 2.0
0.122525670200678 0.109553732585441 1.0
0.957614201352413 0.816338687734049 2.0
0.954475792958789 0.812088991942136 2.0
0.761017425661233 0.729522837207324 0.0
0.672161460721602 0.861002166450453 1.0
0.239406790192071 0.365618025721266 1.0
0.870723011087122 0.327873953455547 2.0
0.934280492689494 0.851254831683878 2.0
0.14797795653602 0.0953888768765728 1.0
0.934506246597952 0.397985510395052 2.0
0.621952713032131 0.126325527799409 0.0
0.897532556013745 0.330227449282505 2.0
0.392694229690084 0.186643904085458 0.0
0.859269668380437 0.839211285626394 2.0
0.401748268395756 0.234889116202258 0.0
0.808932697666782 0.354556134303312 2.0
0.877260236777855 0.335874496099465 2.0
0.678061380229882 0.824244462599356 1.0
0.215927603313165 0.514652194403373 1.0
0.89784447460324 0.322088410978665 2.0
0.0966981224922278 0.328709404902918 2.0
0.646549615546474 0.0157676267499616 0.0
0.179878087083153 0.318909610845724 2.0
0.749211161599117 0.785172765363219 0.0
0.659444142425379 0.718602279270739 0.0
0.665174433745889 0.774557083051278 0.0
0.642320344200454 0.881705789579018 1.0
0.647549623787786 0.748006408935465 0.0
0.613450687706433 0.0760773349120083 0.0
0.194407771884352 0.450645823570935 1.0
0.834861561459554 0.758481425892507 0.0
0.186254108689871 0.342458480237298 2.0
0.106652860437064 0.134330994874222 1.0
0.601755932032472 0.119211043219762 0.0
0.663503768417727 0.725534851220839 0.0
0.323105950621072 0.100122371702216 0.0
0.129374945982643 0.101504164417501 1.0
0.695784240736012 0.731355012744166 0.0
0.735594162502601 0.906265556064715 1.0
0.635260452319226 0.857224103375849 1.0
0.708420454097539 0.734825663934513 0.0
0.917146975581367 0.823852592728448 2.0
0.653673710602702 0.866155267803315 1.0
0.123527151761012 0.369154028409816 2.0
0.922256878367559 0.823992859326356 2.0
0.638688849047178 0.111497237268863 0.0
0.919706967604308 0.365096647452448 2.0
0.890252892772648 0.357296452307252 2.0
0.197024697415178 0.40595096834461 1.0
0.677390113595922 0.764471067720558 0.0
0.0999462419212165 0.0476750558363699 1.0
0.607505548205721 0.985880859813955 1.0
0.640368920513579 0.0316429965082748 0.0
0.11958663579804 0.340660585794129 2.0
0.0814294173753167 0.0964993683753376 1.0
0.629435385797743 0.0340216414539941 0.0
0.800543601179934 0.776638992048084 0.0
0.891220815973608 0.273720366839433 2.0
0.965891189758763 0.906373975968508 2.0
0.386027989042875 0.22683209978988 0.0
0.100023803973526 0.174298922110446 1.0
0.0182651908971974 0.128206218734818 1.0
0.842960602370445 0.851947798061519 2.0
0.213730644466279 0.478983487157708 1.0
0.714778555870976 0.942888181332986 1.0
0.713270868225841 0.873897510201288 1.0
0.375800973067635 0.192114981966569 0.0
0.882131487643561 0.349899488505695 2.0
0.285929147910746 0.455086937641188 1.0
0.791268391410117 0.362813024262253 2.0
};
\end{axis}

\end{tikzpicture}

%% file: main_arxiv.bbl
\begin{thebibliography}{10}

\bibitem{aghaei2020learning}
Sina Aghaei, Andr{\'e}s G{\'o}mez, and Phebe Vayanos.
\newblock Strong optimal classification trees.
\newblock {\em Operations Research}, 73(4):2223--2241, 2025.

\bibitem{writing}
Claus Bahlmann, Bernard Haasdonk, and Hans Burkhardt.
\newblock Online handwriting recognition with support vector machines-a kernel
  approach.
\newblock In {\em Proceedings Eighth International Workshop on Frontiers in
  Handwriting Recognition}, pages 49--54. IEEE, 2002.

\bibitem{baldomero2020tightening}
Marta Baldomero-Naranjo, Luisa~I Mart{\'\i}nez-Merino, and Antonio~M
  Rodr{\'\i}guez-Ch{\'\i}a.
\newblock Tightening big ms in integer programming formulations for support
  vector machines with ramp loss.
\newblock {\em European Journal of Operational Research}, 286(1):84--100, 2020.

\bibitem{baldomero2021robust}
Marta Baldomero-Naranjo, Luisa~I Mart{\'\i}nez-Merino, and Antonio~M
  Rodr{\'\i}guez-Ch{\'\i}a.
\newblock A robust svm-based approach with feature selection and outliers
  detection for classification problems.
\newblock {\em Expert Systems with Applications}, 178:115017, 2021.

\bibitem{benitez2019cost}
Sandra Ben{\'\i}tez-Pe{\~n}a, Rafael Blanquero, Emilio Carrizosa, and Pepa
  Ram{\'\i}rez-Cobo.
\newblock Cost-sensitive feature selection for support vector machines.
\newblock {\em Computers \& Operations Research}, 106:169--178, 2019.

\bibitem{bertsimas2017optimal}
Dimitris Bertsimas and Jack Dunn.
\newblock Optimal classification trees.
\newblock {\em Machine Learning}, 106(7):1039--1082, 2017.

\bibitem{biggs2025tight}
Max Biggs and Georgia Perakis.
\newblock Tight mixed-integer optimization formulations for prescriptive trees.
\newblock {\em Machine Learning}, 114(7):156, 2025.

\bibitem{BlancoJaponPuerto_ADAC2020}
V{\'\i}ctor Blanco, Alberto Jap{\'o}n, and Justo Puerto.
\newblock Optimal arrangements of hyperplanes for svm-based multiclass
  classification.
\newblock {\em Advances in Data Analysis and Classification}, 14(1):175--199,
  2020.

\bibitem{blanco2022mathematical}
V{\'\i}ctor Blanco, Alberto Jap{\'o}n, and Justo Puerto.
\newblock A mathematical programming approach to svm-based classification with
  label noise.
\newblock {\em Computers \& Industrial Engineering}, 172:108611, 2022.

\bibitem{blanco2022robust}
Victor Blanco, Alberto Jap{\'o}n, and Justo Puerto.
\newblock Robust optimal classification trees under noisy labels.
\newblock {\em Advances in Data Analysis and Classification}, 16(1):155--179,
  2022.

\bibitem{blanco2022forests}
V{\'\i}ctor Blanco, Alberto Jap{\'o}n, Justo Puerto, and Peter Zhang.
\newblock A mathematical programming approach to optimal classification
  forests.
\newblock {\em arXiv preprint arXiv:2211.10502}, 2022.

\bibitem{bpr20}
Victor Blanco, Justo Puerto, and Antonio~M Rodriguez-Chia.
\newblock On lp-support vector machines and multidimensional kernels.
\newblock {\em J. Mach. Learn. Res.}, 21:14--1, 2020.

\bibitem{boutilier2023optimal}
Justin Boutilier, Carla Michini, and Zachary Zhou.
\newblock Optimal multivariate decision trees.
\newblock {\em Constraints}, 28(4):549--577, 2023.

\bibitem{bmz-cpaior22}
Justin~J. Boutilier, Carla Michini, and Zachary Zhou.
\newblock Shattering inequalities for learning optimal decision trees.
\newblock In {\em Integration of Constraint Programming, Artificial
  Intelligence, and Operations Research: 19th International Conference, CPAIOR
  2022, Los Angeles, CA, USA, June 20-23, 2022, Proceedings}, page 74–90,
  Berlin, Heidelberg, 2022. Springer-Verlag.

\bibitem{boyd2004convex}
Stephen~P. Boyd and Lieven Vandenberghe.
\newblock {\em Convex Optimization}.
\newblock Cambridge University Press, Cambridge, UK, 2004.

\bibitem{CART}
Leo Breiman.
\newblock {\em Classification and regression trees}.
\newblock Routledge, 2017.

\bibitem{burgard2024mixed}
Jan~Pablo Burgard, Maria~Eduarda Pinheiro, and Martin Schmidt.
\newblock Mixed-integer quadratic optimization and iterative clustering
  techniques for semi-supervised support vector machines.
\newblock {\em Top}, 32(3):391--428, 2024.

\bibitem{collins2002logistic}
Michael Collins, Robert~E Schapire, and Yoram Singer.
\newblock Logistic regression, adaboost and bregman distances.
\newblock {\em Machine Learning}, 48(1-3):253--285, 2002.

\bibitem{svm}
Corinna Cortes and Vladimir Vapnik.
\newblock Support-vector networks.
\newblock {\em Machine learning}, 20(3):273--297, 1995.

\bibitem{cs}
Koby Crammer and Yoram Singer.
\newblock On the algorithmic implementation of multiclass kernel-based vector
  machines.
\newblock {\em Journal of machine learning research}, 2(Dec):265--292, 2001.

\bibitem{cristianini2000introduction}
Nello Cristianini and John Shawe-Taylor.
\newblock {\em An Introduction to Support Vector Machines and Other
  Kernel-based Learning Methods}.
\newblock Cambridge University Press, Cambridge, UK, 2000.

\bibitem{d2024margin}
Federico D’Onofrio, Giorgio Grani, Marta Monaci, and Laura Palagi.
\newblock Margin optimal classification trees.
\newblock {\em Computers \& Operations Research}, 161:106441, 2024.

\bibitem{fisher1936use}
Ronald~A Fisher.
\newblock The use of multiple measurements in taxonomic problems.
\newblock {\em Annals of eugenics}, 7(2):179--188, 1936.

\bibitem{freed1986evaluating}
Ned Freed and Fred Glover.
\newblock Evaluating alternative linear programming models to solve the
  two-group discriminant problem.
\newblock {\em Decision Sciences}, 17(2):151--162, 1986.

\bibitem{freund1997decision}
Yoav Freund and Robert~E Schapire.
\newblock A decision-theoretic generalization of on-line learning and an
  application to boosting.
\newblock {\em Journal of computer and system sciences}, 55(1):119--139, 1997.

\bibitem{gunluk2018optimal}
Oktay G{\"u}nl{\"u}k, Jayant Kalagnanam, Minhan Li, Matt Menickelly, and Katya
  Scheinberg.
\newblock Optimal decision trees for categorical data via integer programming.
\newblock {\em Journal of Global Optimization}, 81(1):233--260, 2021.

\bibitem{credit}
Terry Harris.
\newblock Quantitative credit risk assessment using support vector machines:
  Broad versus narrow default definitions.
\newblock {\em Expert Systems with Applications}, 40(11):4404--4413, 2013.

\bibitem{insurance}
Vladimir Ka{\v{s}}{\'c}elan, Ljiljana Ka{\v{s}}{\'c}elan, and Milijana
  Novovi{\'c}~Buri{\'c}.
\newblock A nonparametric data mining approach for risk prediction in car
  insurance: a case study from the montenegrin market.
\newblock {\em Economic research-Ekonomska istra{\v{z}}ivanja}, 29(1):545--558,
  2016.

\bibitem{labbe2019mixed}
Martine Labb{\'e}, Luisa~I Mart{\'\i}nez-Merino, and Antonio~M
  Rodr{\'\i}guez-Ch{\'\i}a.
\newblock Mixed integer linear programming for feature selection in support
  vector machine.
\newblock {\em Discrete Applied Mathematics}, 261:276--304, 2019.

\bibitem{llw}
Yoonkyung Lee, Yi~Lin, and Grace Wahba.
\newblock Multicategory support vector machines: Theory and application to the
  classification of microarray data and satellite radiance data.
\newblock {\em Journal of the American Statistical Association},
  99(465):67--81, 2004.

\bibitem{cancer}
Abdul Majid, Safdar Ali, Mubashar Iqbal, and Nabeela Kausar.
\newblock Prediction of human breast and colon cancers from imbalanced data
  using nearest neighbor and support vector machines.
\newblock {\em Computer methods and programs in biomedicine}, 113(3):792--808,
  2014.

\bibitem{mangasarian1999arbitrary}
Olvi~L Mangasarian.
\newblock Arbitrary-norm separating plane.
\newblock {\em Operations Research Letters}, 24(1-2):15--23, 1999.

\bibitem{cleveland}
S~Radhimeenakshi.
\newblock Classification and prediction of heart disease risk using data mining
  techniques of support vector machine and artificial neural network.
\newblock In {\em 2016 3rd International Conference on Computing for
  Sustainable Global Development (INDIACom)}, pages 3107--3111. IEEE, 2016.

\bibitem{scholkopf2002learning}
Bernhard Schölkopf and Alexander~J. Smola.
\newblock {\em Learning with Kernels: Support Vector Machines, Regularization,
  Optimization, and Beyond}.
\newblock MIT Press, Cambridge, MA, 2002.

\bibitem{vapnik1982estimation}
Vladimir Vapnik.
\newblock Estimation of dependences based on empirical data: Springer series in
  statistics (springer series in statistics), 1982.

\bibitem{ww}
Jason Weston and Chris Watkins.
\newblock Support vector machines for multi-class pattern recognition.
\newblock In {\em Esann}, volume~99, pages 219--224, 1999.

\end{thebibliography}
